\documentclass[11pt]{amsart}
\usepackage[letterpaper,margin=1in]{geometry}                
\usepackage{amssymb,amsmath,amsfonts}
\usepackage{xcolor}
\usepackage[%colorlinks,
             pdfproducer={pdfLaTeX},
             pdfpagemode=None,
             bookmarksopen=true
             bookmarksnumbered=true]{hyperref}
\usepackage{graphicx}
\usepackage[utf8]{inputenc}

\newtheorem{dummy}{dummy}[subsection]
\newtheorem{theorem}[dummy]{Theorem}
\newtheorem{proposition}[dummy]{Proposition}
\newtheorem{conjecture}[dummy]{Conjecture}
\newtheorem{corollary}[dummy]{Corollary}
\newtheorem{lemma}[dummy]{Lemma}
\newtheorem{definition}[dummy]{Definition}

\theoremstyle{definition}

\newtheorem{example}[dummy]{Example}
\newtheorem{remark}[dummy]{Remark}

\newcommand\MRnumber[2]{}

\newcommand\cC{\mathcal C}
\newcommand\cD{\mathcal D}

\newcommand\cV{\mathcal V}

\newcommand\cX{\mathcal X}
\newcommand\cY{\mathcal Y}

\newcommand\bC{\mathbb C}

\newcommand\bK{\mathbb K}
\newcommand\bN{\mathbb N}

\newcommand\bR{\mathbb R}

\newcommand\rB{\mathrm B}

\usepackage[utf8]{inputenc}
\DeclareFontFamily{U}{min}{}
\DeclareFontShape{U}{min}{m}{n}{<-> udmj30}{}
\newcommand\yo{\!\text{\usefont{U}{min}{m}{n}\symbol{'207}}\!}

\newcommand\condense{\mathrel{\,\hspace{.75ex}\joinrel\rhook\joinrel\hspace{-.75ex}\joinrel\rightarrow}}

\newcommand\from\leftarrow
\newcommand\mono\hookrightarrow
\newcommand\epi\twoheadrightarrow

\newcommand\<\langle
\renewcommand\>\rangle
\newcommand\sminus\smallsetminus

\DeclareMathOperator\id{id}

\DeclareMathOperator\Kar{Kar}
\DeclareMathOperator\End{End}

\DeclareMathOperator\Mat{Mat}
\DeclareMathOperator\Psh{Psh}
\DeclareMathOperator\Cau{Cau}

\newcommand\Vect{\cat{Vec}}
\newcommand\SVect{\cat{SVec}}

\newcommand\Cat{\cat{Cat}}

\DeclareMathOperator*\colim{colim}

\newcommand\define[1]{\emph{#1}}
\newcommand\cat[1]{\mathbf{#1}}

\usepackage{tikz}
\usetikzlibrary{decorations.markings,arrows,decorations.pathreplacing,arrows,decorations.pathmorphing,decorations.pathreplacing,decorations.markings,shapes.geometric,through,fit,shapes.symbols}
\tikzset{ribbon/.style={draw=white,double=purple,very thick,double distance=2pt}}
\tikzset{onearrowlate/.style={postaction={decorate}, decoration={markings,mark=at position .7 with {\arrow[draw,line width=1pt]{>}}}}}
\tikzset{onearrow/.style={postaction={decorate}, decoration={markings,mark=at position .55 with {\arrow[draw,line width=1pt]{>}}}}}
\tikzset{condensation/.style={-,
          decoration={markings, 
                    mark=at position -1ex with {\arrow{left hook}},
                    mark=at position 1 with {\arrow[scale=1.5]{>}},
                   },
        preaction={decorate}}}
\tikzset{condensationlarge/.style={-,
          decoration={markings, 
                    mark=at position -1.5ex with {\arrow[scale=1.25]{left hook}},
                    mark=at position 1 with {\arrow[scale=1.75]{>}},
                   },
        preaction={decorate}}}
\tikzset{doublecondensation/.style={double,
        double distance=1.5pt,
        shorten <=6pt, shorten >=6pt, 
        decoration={markings, 
                    mark=at position -11pt with {\arrow[scale=1.25]{left hook}},
                    mark=at position -4.25pt with {\arrow[scale=1.5]{>}},
                   },
        preaction={decorate}}}
\tikzset{doublecondensationlarge/.style={double,
        double distance=1.5pt,
        shorten <=6pt, shorten >=6pt, 
        decoration={markings, 
                    mark=at position -11pt with {\arrow[scale=1.5]{left hook}},
                    mark=at position -4.25pt with {\arrow[scale=2]{>}},
                   },
        preaction={decorate}}}
\tikzset{
    dot/.style={circle,draw,fill,inner sep=1pt},
    arrow/.style={->,thick,shorten <=2pt,shorten >=2pt},
    twoarrow/.style={double,double distance=1.5pt,shorten <=9pt,shorten >=10pt,decoration={markings,mark=at position -8pt with {\arrow[scale=1.75]{>}}},preaction={decorate}},
    twoarrowlonger/.style={double,double distance=1.5pt,shorten <=5pt,shorten >=6pt,decoration={markings,mark=at position -4pt with {\arrow[scale=1.75]{>}}},preaction={decorate}},
    twoarrowshorter/.style={double,double distance=1.5pt,shorten <=13pt,shorten >=14pt,decoration={markings,mark=at position -12pt with {\arrow[scale=1.75]{>}}},preaction={decorate}},
    twoarrowshorthead/.style={double,double distance=1.5pt,shorten <=9pt,shorten >=20pt,decoration={markings,mark=at position -18pt with {\arrow[scale=1.75]{>}}},preaction={decorate}},
    threearrowpart1/.style={ thick,double,double distance=3pt,shorten <=9pt,shorten >=11pt},
    threearrowpart2/.style={ thick,shorten <=9pt,shorten >=10pt},
    threearrowpart3/.style={ shorten <=9pt,shorten >=10pt,decoration={markings,mark=at position -8pt with {\arrow[scale=3]{>}}},preaction={decorate}},
fourarrowpart1/.style={thick, double,double distance=4pt,shorten <=1pt,shorten >=2.75pt},
fourarrowpart2/.style={thick,  double,double distance=1pt,shorten <=1pt,shorten >=1.25pt,decoration={markings,mark=at position -.05pt with {\arrow[scale=3,ultra thin]{>}}},preaction={decorate} }
}

\pgfdeclarelayer{background}
\pgfdeclarelayer{foreground}
\pgfsetlayers{background,main,foreground}

\newcommand{\be}{\begin{equation*}}
\newcommand{\ee}{\end{equation*}}

\title{Condensations in higher categories}
\author{Davide Gaiotto and Theo Johnson-Freyd}
\thanks{Research at the Perimeter Institute is supported by the Government of Canada through Industry Canada and by the Province of Ontario through the Ministry of Economic Development and Innovation. The Perimeter Institute is in the Haldimand Tract, land promised to the Six Nations. The first-named author is supported by grants from the Krembil Foundation and the NSERC Discovery Grant program. The second-named author is supported by grants from the Simons Foundation and the NSERC Discovery Grant program. We thank David Reutter and Claudia Scheimbauer for suggestions and discussions, and we thank the anonymous referees for their detailed comments.}

\begin{document}

\begin{abstract}
We present a higher-categorical generalization of the ``Karoubi envelope'' construction from ordinary category theory, and argue that, like the ordinary Karoubi envelope, our higher Karoubi envelope is the closure for absolute limits. Our construction replaces the idempotents in the ordinary version with a notion that we call ``condensations.'' The name is justified by the direct physical interpretation of the notion of condensation: it encodes a general class of constructions which produce a new topological phase of matter by turning on a commuting projector Hamiltonian on a lattice of defects within a different topological phase (for example, the trivial phase). We also identify our higher Karoubi envelopes with categories of fully-dualizable objects. Together with the Cobordism Hypothesis, we argue that this realizes an equivalence between a very broad class of gapped topological phases of matter and fully extended topological field theories, in any number of dimensions. 
\end{abstract}
\maketitle

\tableofcontents

\section{Introduction}

Karoubi aka idempotent completion, which adjoins to a category images for all idempotent morphisms, is a vital tool for the construction and analysis of interesting categories. This paper describes the corresponding construction for weak $n$-categories. Our motivation comes from physics: we find that ``condensation'' of gapped topological quantum matter systems is precisely axiomatized by forming the image of a higher-categorical idempotent.
 As such, we will use the name \define{condensation} for our higher-categorical analog of forming images of idempotents, and \define{condensation monad} for the higher-categorical idempotents themselves,  whether the $n$-category under consideration is or is not of physical origin. Section~\ref{intro.gappedvstqft} explains in more detail the physics motivating our construction, and \S\ref{intro.layering} continues the physics story, building up to our inductive definition of \define{(categorical) condensation}. We slowly migrate to more mathematical language in \S\ref{subsec.intro-Kar}, where we survey the mathematical results of the paper. The paper itself is pure category theory, with the exception of \S\ref{sec.bc} on lattice Hamiltonians and \S\ref{subsec.stateoperatorcorresp} on the state-operator correspondence. A reader who is allergic to physics could in principle begin in Section~\ref{sec.Kar}, although they would miss some important ideas and motivation.

\begin{remark}\label{rem:modelindependent}
  Weak $n$-categories remain under active development, and basic results about the equivalence between different models have not been entirely established. Rather than work in any specific model, our main definitions and constructions are so robust that they are easily implemented in any model of weak $n$-categories. Our proofs of Theorem~\ref{thm.absolute} and Corollary~\ref{cor.dualizable}, which are the strongest and most important results in this paper, require  assumptions about weak $n$-categories, and specifically about colimits therein, that we could not locate in the literature. These assumptions are described in \S\ref{subsec.abcolim}.
  
  Because we do not implement our results in any specific model, although we will write in terms of Theorems and Proofs, our results should be understood as having the mathematical status of detailed sketches. Our goal is to explain the general higher dimensional picture, and we invite others to continue to develop the story. For example, Markus Zetto is currently investigating higher Karoubi completion in the general setting of enriched higher categories, and our understanding is that he will provide a fully-implemented version.
  
  More strongly, we suggest our results as a desideratum for good models of weak $n$-categories: a proposed model for which our constructions and proofs do not compile should be rejected as not capturing what end-users mean by ``weak $n$-category.''
  Experts will moreover recognize that our notion of condensation makes sense, with only slight modification, not just in the world of weak $n$-categories, but also in the world of $(\infty,n)$-categories. We did not attempt to pursue that generality: instead, we invite those experts to build on our results. 
  \end{remark}

\subsection{Gapped phases of matter vs extended TQFTs}\label{intro.gappedvstqft}

In recent years, the study of ``topological'' gapped phases of matter \cite{ChenGuWen,KitaevTalk3} has considerably expanded the physical range of applicability of the 
tools of Topological Quantum Field Theory.

A quantum system is \define{gapped} when the creation of excitations above the vacuum costs a non-zero amount of energy. Almost by definition, TQFTs provide a low-energy effective description of gapped Quantum Field Theories defined in the continuum. 
For comparison,  it is far from obvious that TQFTs should provide 
the correct language to describe the long-distance, low-energy 
behaviour of generic gapped condensed matter systems, which are typically defined as some local lattice Hamiltonian and do not have built-in relativistic invariance. 
Nevertheless, it is an empirical fact that many gapped condensed matter systems which  admit a robust long-distance/low-energy limit are  well-described by TQFTs.

\begin{remark}
It is certainly possible to find gapped condensed matter systems which do  not admit a TQFT description --- there are gapped condensed matter systems in which global properties  depend sensitively on small local changes of the system \cite{Vijay:2015mka}.  
\end{remark}

This relationship between gapped condensed matter systems and TQFTs is perplexing, particularly so if 
one takes a ``global'' approach to TQFTs, defining them a la \cite{At} in terms of 
partition functions attached to non-trivial Euclidean space-time manifolds and spaces of states attached to 
non-trivial space manifolds. From that perspective, matching a given lattice system to a TQFT would require 
identifying a lot of extra structure to be added to the definition of the lattice system in order to define it on discretizations of 
non-trivial space manifolds and to define adiabatic evolutions analogous to non-trivial space-time manifolds.

\begin{remark}If one takes the information-theoretic perspective that a phase of matter is fully characterized by the local entanglement properties of the 
ground-state wavefunction of the system \cite{Kitaev:2005dm,2006PhRvL..96k0405L}, with no reference to a time evolution, then even more work may be needed.\end{remark}

One should then worry that this extra structure may hide choices which would complicate the correspondence 
between gapped phases and TQFTs. Conversely, given a TQFT we may or not be able to provide a microscopic lattice realization which ``discretizes'' it. There could be topological field theories, or field theories in general, 
which are well defined and yet cannot be put on a lattice, or can be put on a lattice in different, inequivalent ways. 
The notion of ``phase'' in condensed matter system is quite strict and multiple distinct phases could potentially
look the same through the lens of topological field theory. 

These concerns are not just abstract. Given some phase of matter, in the lab or in a computer,
it is hard to extract the data which would pin down the corresponding TQFT, or even know if the TQFT exists. 
For example, we can hardly place a three-dimensional material on a non-trivial space-manifold. We can only try to simulate that 
by employing judicious collections of defects in flat space. 

%\subsection{Gapped phases of matter vs extended TQFTs}\label{intro.gappedvsetqft}

Some of the tension is ameliorated by a more modern perspective on TQFT, which stresses the local properties of the theory, encoded in  local defects of all possible codimension. The so-called (Fully) Extended TQFTs are defined by an intricate 
system of mathematical data attached to an open ball. The ``Cobordism Hypothesis'' 
of \cite{BaeDol95,Lur09}
guarantees that these data 
can be used to systematically assemble partition functions, Hilbert spaces, etc., for complicated manifolds 
out of much simpler geometric building blocks. 
In particular, in order to match some physical system to a fully extended TQFT we would not need to consider 
nontrivial spacetime geometries. We would instead have to study a sufficiently rich collection of local modifications 
of the physical system.

The Cobordism Hypothesis is essentially a statement about topology. It gives the tools to reconstruct something which behaves as a  TQFT from a ``fully dualizable'' object in some higher category. It leaves open the problem of identifying 
the correct choice of higher category to produce a given class of physical TQFTs:
\begin{itemize}
\item Which higher category classifies extended TQFTs associated to $d$-dimensional gapped lattice systems built from bosonic degrees of freedom? \item What about $d$-dimensional gapped lattice systems built from fermionic degrees of freedom?
\item What about systems enhanced by $G$-symmetry?
\item Etc.
\end{itemize}

Furthermore, the details of how a higher category is presented matter: they control which local questions one would need to ask about a physical system in order to reconstruct the corresponding extended TQFT.  From this ``practical'' perspective, the current state of the field is somewhat dismal. 
Higher categories with interesting fully dualizable objects are only available in low dimension and the 
presentation of the categories is not well suited to a comparison to phases of matter. 
Typical $d$-categories are presented in terms of local operators: the objects are by definition the ``theories'' defined in $d$ space-time dimensions, morphisms are by definition the codimension-$1$ defects between theories, 2-morphisms are codimension-$2$ defects, etc. Hilbert spaces of states are a derivative notion, built essentially by the state-operator map of 
TQFTs. Full dualizability is roughly a requirement for the state-operator map to be well-defined. 

In order to make contact with condensed matter constructions, a better approach would be if all ingredients were cast directly in terms of Hilbert spaces of low-energy 
states. Notions like ``local operator'' and ``defect'' could then be derived in such a way that the state-operator map and full dualizability 
will hold automatically. This is the physical objective in this paper. Specifically, our goal is to define some higher categories in which:
\begin{enumerate}
\item Full dualizability is automatic. \label{goal1}
\item Commuting-projector lattice descriptions are built in. \label{goal2}
\item One can algorithmically answer the question ``Is such-and-such gapped system described by an extended TQFT valued in this higher category?''
\end{enumerate}

\subsection{Layering and condensation} \label{intro.layering}

We will construct our higher categories recursively, building the category of $(d+1)$-dimensional phases out of the category of $d$-dimensional phases. Our construction is a mathematical analogue of the ``coupled layers'' construction \cite{2010PhRvB..81d5120H,2013PhRvB..87w5122W},
which builds a $(d+1)$-dimensional system 
as a stack of $d$-dimensional systems deformed by some local interactions between neighbouring stacks.
Indeed, imagine a stack of many copies of a $d$-dimensional phase $E$. Suppose that we have some way of deforming each consecutive pair of layers $E \otimes E$ into a single copy of $E$, together with higher coherence data. (We will need to work out what that data is!) Then at long distances this interacting stack will determine a new, potentially nontrivial, $(d+1)$-dimensional phase. For our higher category of $(d+1)$-dimensional phases, we will take all phases that can be built in such a ``coupled layers'' way, such that the layers themselves come from our already-constructed higher category of $d$-dimensional phases.
Answering goals~\ref{goal1} and~\ref{goal2} listed above, our Theorem~\ref{thm.fullrigidity} will imply that all objects in our categories are fully dualizable, and so determine extended TQFTs, and Theorem~\ref{thm.Hamiltonian} will construct from each object a commuting projector Hamiltonian lattice system corresponding to the TQFT. 
We expect (but do not prove) the converse:
\begin{conjecture}
  Every commuting projector Hamiltonian lattice system with a TQFT low-energy description arises from the construction in Theorem~\ref{thm.Hamiltonian}.
\end{conjecture}

The presentation of a system as a lattice with a commuting projector Hamiltonian is a special case of a more general 
setup: consider some $(d+1)$-dimensional phase $X$, decorate it with some mesoscopic lattice of lower-dimensional defects
and turn on a small commuting projector Hamiltonian on top of the Hamiltonian used to define $X$ and the defects, so that 
at long distances and low energies the resulting system flows to a new $(d+1)$-dimensional phase $Y$. 
The layering construction we have in mind is simply the mesoscopic lattice construction in which the starting phase $X$ is the trivial vacuum phase, and the $d$-dimensional phase $E$ is the defect used to decorate the $d$-dimensional walls of the lattice.

A mesoscopic lattice construction starting with a phase $X$ and producing a phase $Y$ deserves to be called a ``condensation,'' in analogy to the anyon condensation in low dimension \cite{2018ARCMP...9..307B}. Pick some way of ``condensing'' a phase $X$ into a new phase $Y$ (either presented by a mesoscopic lattice or some other way). Imagine flooding a room with phase $X$ and condensing just half of the room to phase $Y$. Then there will be an interface between $X$ and $Y$, and the physics of that interface is determined by the choice of condensation procedure.  Let us say that we have a \define{gapped condensation} when $X$ and $Y$ and also the interface are all gapped topological.

\begin{remark} \label{rem.framingdependence1}
  We will not require any a priori rotational invariance of our physical systems. Microscopic condensed matter models never have continuous rotation symmetry, since they are typically built on lattices; any rotation invariance is emergent only in the low-energy, and must be established by theorem, not fiat. 
  This manifests in the low-energy limit as a possible framing dependence: condensed matter naturally produces framed, not oriented, TQFTs
  
  By the same token, ``the'' interface produced by condensing half of the room could very much depend on which half of the room is condensed: even if $X$ and $Y$ have (emergent) rotational invariance, so that it makes sense to rotate an interface between them, we will not assume that the interface produced by condensation is rotationally-invariant. Rather, for a gapped condensation, the interface produced between $X$ and $Y$ should be gapped regardless of its angle.
\end{remark}

Now imagine that we have taken our $X$-filled room and condensed both the left and right halves of the room to phase $Y$, leaving a region of phase $X$ in the middle. Reading along, we see: the ($(d+1)$-dimensional) phase $Y$, a ($d$-dimensional) interface ``$g$'' from $Y$ to $X$, phase $X$, an interface ``$f$'' from $X$ to $Y$, and finally more phase $Y$. Assuming these phases and interfaces truly are gapped topological, we can imagine ``composing'' the interfaces $f$ and $g$ into a single ($d$-dimensional) interface $f\circ g$ from $Y$ to itself. 

\begin{remark}
This physical discussion requires important assumptions about the 
gapped systems we are working with: we are assuming that local defects can be defined individually and moved with respect 
to each other without drastic jumps in the properties of the system. In other words, it assumes that there is a meaningful higher category whose objects are $(d+1)$-dimensional gapped topological phases, and whose $k$-morphisms are $(d-k+1)$-dimensional interfaces between $(k-1)$-morphisms. 
 These assumptions require us to work with 
gapped systems and defects which admit a reasonable long-distance/low-energy limit.

Reversing the logic, in \S\ref{sec.bc} we will construct higher categories consisting of (certain) gapped topological phases. These to-be-constructed higher categories will answer the requests from the end of \S\ref{intro.gappedvstqft}.
\end{remark}

Now complete the condensation procedure and convert the whole $X$-filled region to phase $Y$. Equivalently, this effects a gapped condensation of interfaces from $f \circ g$ onto the invisible ``identity'' interface for $Y$. Thus we have found an inductive algebraic characterization of gapped condensation:
\begin{definition}\label{defn.condensation-defect}
  Given two gapped systems $X,Y$, or two local defects in a higher-dimensional system, a \define{categorical condensation of $X$ onto $Y$} is a pair of gapped interfaces $f$ from $X$ to $Y$ and $g$ from $Y$ to $X$, together with a condensation of the composite interface $f \circ g$ from $Y$ to $Y$ onto the trivial interface $\id_Y$ from $Y$ to $Y$.
\end{definition}
The induction implicit in Definition~\ref{defn.condensation-defect} is on the dimension of the phase or interface: categorical condensations between the $(d+1)$-dimensional objects $X$ and $Y$ are defined in terms of categorical condensations between the $d$-dimensional objects $f\circ g$ and $\id_Y$. 
To begin the induction, we declare that identity interfaces are examples of condensations. Theorem~\ref{thm.Hamiltonian} will show that a categorical condensation provides enough information to write down a mesocopic lattice model, so that in fact the passage from a physical gapped condensation to the corresponding categorical condensation involves no loss of information. With this in mind, we will generally omit the word ``categorical,'' and use \define{condensation} to refer to the inductively-defined data of  Definition~\ref{defn.condensation-defect}.

\begin{example}
Suppose a system $X$ is equipped with a nonanomalous action by some finite group $G$, and let $Y$ be obtained by coupling $X$ to dynamical $G$ 
gauge fields, i.e.\ by ``gauging'' the $G$ symmetry of~$X$. Then $X$ condenses onto $Y$ in the sense of Definition~\ref{defn.condensation-defect}. 

Indeed, the interfaces 
$f$ and $g$ can both be taken to be the trivial interface on the $X$-degrees of freedom, combined with a Dirichlet boundary condition for the $G$ gauge fields
(the Dirichlet boundary condition freezes the dynamical $G$-connection to the trivial one).
 A basic property of Dirichlet boundary conditions is that the composition $f \circ g$ automatically carries its own non-anomalous $G$-action, and that gauging that action produces 
  the trivial interface $\id_Y$. This shows, inductively, that $X$ condenses onto $Y$. 
  
  At the bottom of the induction we have a ($(0+1)$-dimensional) anyon $X_1$ and a group homomorphism $G \to \operatorname{Aut}(X_1)$. To gauge this action is simply to take the $G$-fixed sub-anyon $Y_1 \subset X_1$. This sub-anyon is selected by the projector $\frac{1}{|G|} \sum_{g \in G} g \in \End(X_1)$, and comes with inclusion and projection maps $Y_1 \overset{g_1}\mono X_1 \overset{f_1}\epi Y_1$ such that $f_1 \circ g_1 = \id_{Y_1}$. See Example~\ref{eg.anyoncondensation}.
\end{example}

\subsection{Higher Karoubi envelopes}\label{subsec.intro-Kar}

Definition~\ref{defn.condensation-defect} also makes sense in any weak $n$-category: one simply replaces the words ``gapped system''  and ``interface'' with morphisms of appropriate dimension. Specifically, given a weak $n$-category $\cC$, and objects $X,Y \in \cC$, a \define{condensation} of $X$ onto $Y$ is a pair of $1$-morphisms $f:X\to Y$ and $g : Y \to X$ together with a condensation (in the weak $(n-1)$-category $\End_\cC(Y)$) of $f\circ g$ onto $\id_Y$.
In the special case that $\cC$ is a $1$-category, a condensation in~$\cC$ is the same as what is variously called a ``split surjection'' or a ``retract.'' There does not seem to be a standard notation for this well-used notion, and so we propose ``$X \condense Y$,''%
\footnote{\LaTeX\ code: \texttt{%
{\char`\\}mathrel\{\char`\\,{\char`\\}hspace\{.75ex\}{\char`\\}joinrel{\char`\\}rhook{\char`\\}joinrel{\char`\\}hspace\{-.75ex\}{\char`\\}joinrel{\char`\\}rightarrow\}
}}
 with the notation intended to suggest both  $f : X \epi Y$ and $X \hookleftarrow Y : g$.
 One of our main mathematical results, Theorem~\ref{thm.absolute}, says that condensations are the ``correct'' higher-categorical generalization  of split surjections.

In a linear 1-category, a condensation $X \condense Y$ writes $X$ as a direct sum $X \cong Y \oplus \cdots$ --- the direct summand $Y$ is realized as the image of the idempotent $e = g\circ f$. (In a 1-category, an endomorphism $e : X \to X$ is \define{idempotent} if $e^2=e$.)
This is in particular the case if $X$ and $Y$ are anyons, aka point-like defects in space or line-defects in space-time. For higher-dimensional~$X$ and~$Y$, unpacking Definition~\ref{defn.condensation-defect} leads to a collection of morphisms of various dimensions, all the way to the top.

For $X$ and $Y$ of arbitrary dimension, the interface $e := g \circ f$ has several nice properties, generalizing the equation $e^2=e$ when $X$ and $Y$ are anyons. 
We will axiomatize these properties in \S\ref{subsec.condmon} under the name \define{condensation monad}.
For example, we get for free a condensation  $e \circ e \condense e$. It is not hard to argue that one should be able to recover $Y$ by deforming $X$ with a regular array of $e$ 
interfaces, junctions between $e$ and $e \circ e$ interfaces, etc., all the way down to commuting projectors. 
Of particular interest is the case when $X$ is the $(d+1)$-dimensional vacuum/trivial phase. Then $e$ is itself a $d$-dimensional phase, and we have found a construction of the $(d+1)$-dimensional phase~$Y$ as a layered stack of $e$-phases. 
More generally, we will produce, given any condensation monad $e \in \End(X)$, a new system $Y$ built out of a lattice of $e$-defects in $X$.
The construction $e \leadsto Y$ is a higher-categorical version of forming the image of an idempotent. We will call $Y$ the \define{condensate} of the condensation monad~$e$.

%\subsection{Karoubi envelope}\label{subsec.intro-Kar}

In the theory of 1-categories, there is a well-known construction called the ``Karoubi envelope'' which inputs a 1-category $\cC$ and produces a new 1-category $\Kar(\cC)$. The $\Kar(-)$ construction is totally elementary: the objects of $\Kar(\cC)$ are the idempotents in $\cC$; the morphisms are given by similarly-simple diagrams in $\cC$. But it is also powerful: $\Kar(\cC)$ is the universal category which contains $\cC$ and for which every idempotent has an image.
The main mathematical
contribution of this paper is an analogous $\Kar(-)$ construction for weak $n$-categories. Specifically, in Theorem~\ref{thm.Kar} we will produce, given a weak $n$-category $\cC$, a new weak $n$-category $\Kar(\cC)$ whose objects are the condensation monads in $\cC$, which is universal for the property that $\Kar(\cC)$ contains $\cC$ and includes all condensates. In terms of the above physical discussion, if $\cC$ was some collection of gapped phases and the defects between them, then $\Kar(\cC)$ consists of all phases that can be produced from the phases in $\cC$ by condensation.

\begin{remark}
  One of the main results of \cite{Reutter2018} is a version of the Karoubi envelope for linear $2$-categories satisfying some mild finiteness conditions. Although the details of their construction differ from ours (most notably, their analogs of our ``condensations'' and ``condensation monads'' include a unitality condition; see \S\ref{sec.unitalcond}), we will show in Theorem~\ref{thm.reutter} that the outputs agree.
\end{remark}

Another basic categorical construction is to take a monoidal weak $n$-category $(\cC,\otimes)$ and produce a weak $(n+1)$-category with one object~``$\bullet$'' and endomorphisms $\End(\bullet) = \cC$, with composition given by $\otimes$; the standard name for this weak $(n+1)$-category of ``$\rB\cC$.'' If $\cC$ is a collection of $d$-dimensional phases (with $\otimes$ given by the stacking operation of phases), then one should think of the object $\bullet \in \rB\cC$ as the trivial/vacuum $(d+1)$-dimensional phase, and the morphisms in $\rB\cC$ are the phases in $\cC$ thought of as $d$-dimensional defects in the vacuum. 
Then $\Kar(\rB\cC)$ is nothing but the category of all $(d+1)$-dimensional phases which can be produced from layers of phases in $\cC$!

Given its importance, we will denote $\Kar(\rB\cC)$ by the name $\Sigma\cC$.  Iterating the construction produces ever-higher categories $\Sigma^d\cC$.  The main example is:
\begin{example}\label{eg.bosonicgappedphases}
Write $\Vect_\bC$ for the category of finite-dimensional vector spaces over $\bC$; it is well-known that the objects of $\Vect_\bC$ precisely classify bosonic gapped  topological $(0+1)$-dimensional phases.
 The $(d+1)$-category $\Sigma^{d}\Vect_\bC$ then has an interpretation as the collection of all $(d+1)$-dimensional phases which arise from (iterative) condensation. By Theorem~\ref{thm.Hamiltonian}, all phases in $\Sigma^d\Vect_\bC$ admit commuting projector Hamiltonian models. 

By Corollary~\ref{cor.sigmavec}, which encodes a version of the ``state-operator map,'' $\Sigma\Vect_\bC$ is equivalent to the $2$-category of finite-dimensional separable algebras and their finite-dimensional bimodules.
We furthermore expect that $\Sigma^2\Vect_\bC$ is equivalent to the $3$-category of multifusion categories and their (finite semisimple) bimodule categories (see Conjecture~\ref{conj.sigma2vec}; over a general field, we expect to find the 3-category of separable finite tensor categories). These are standard models for categories of $(1+1)$- and $(2+1)$-dimensional phases, respectively. 

Essentially the same statements hold for fermionic phases: one simply replaces $\Vect_\bC$ with the $1$-category $\SVect_\bC$ of finite-dimensional super vector spaces. Similarly, one can replace $\Vect_\bC$ by the category $\cat{Rep}(G)$ of finite-dimensional $G$-modules, for any group $G$, in which case $\Sigma^{n-1}\cat{Rep}(G)$ consists of $G$-symmetry-enriched topological phases.
\end{example}

\begin{remark}
Gapped topological $(2+1)$-dimensional phases that are well-described by multifusion categories are variously called ``Turaev--Viro,'' ``Levin--Wen,'' or ``Barrett--Westbury'' type. There are also gapped topological $(2+1)$-dimensional phases of ``Reshetikhin--Turaev type'' which are not well-described by multifusion categories. 
Together with Conjecture~\ref{conj.sigma2vec}, this illustrates that we do not expect $\Sigma^d\Vect_\bC$ to catch all $(d+1)$-dimensional phases. Rather, it catches those phases that can be built from layering/condensation. Reshetikhin--Turaev TQFTs that are not secretly of Turaev--Viro type cannot by built in this way.
%
%Given a Reshetikhin--Turaev TQFT $\cF$, one may take the fusion category of anyons in $\cF$ and build from it a Turaev--Viro TQFT. The result ends up being isomorphic to $\cF \otimes \bar{\cF}$, where $\bar{\cF}$ denotes the orientation-reversal of $\cF$. In TQFT language, $\cF \otimes \bar{\cF}$ arises as the compactification $\cF[S^0]$ of $\cF$ along a $0$-dimensional sphere. The compactifications $\cF[S^1]$ and $\cF[S^2]$ are, respectively, $(1+1)$- and $(0+1)$-dimensional TQFTs that can also be explicitly realized in terms of condensation. In summary, for $\cF$ an arbitrary $(2+1)$-dimensional TQFT and $d\leq 2$, the compactification $\cF[S^d]$ arises naturally from an object in $\Sigma^{2-d}\Vect_\bC$.
%This motivates:
\end{remark}

%\begin{conjecture}
%  Let $\cF$ be an arbitrary bosonic $(D+1)$-dimensional topological order. For all $d \leq D$, the $((D-d)+1)$-dimensional topological order $\cF[S^d]$ arises from an object in $\Sigma^{D-d}\Vect_\bC$. (As explained in Remark~\ref{rem.framingdependence1}, it is unreasonable to expect topological orders to possess any microscopic rotational invariance. For topological orders, this non-invariance manifests as a possible framing-dependence. The specific $S^d$ along which to compactify is the one with its framing as the boundary of a $(d+1)$-disk.)
%\end{conjecture}
%
%

Completing our answer to the issues raised in \S\ref{intro.gappedvstqft}, we see that the data needed to realize a gapped phase $Y$ (whether a gapped phase of matter or a gapped phase of (T)QFTs) as an object of $\Sigma^{d}\Vect_\bC$ is simply a condensation from the trivial phase to $Y$, which is to say we would need to produce the physical data described above: identify gapped boundary conditions $f$ and $g$, junctions between the interface $f \circ g$ and the identity interface, 
etc. This may be a formidable task, but it is no worse than, say, the problem of identifying 
all the anyons of a $(2+1)$-dimensional system and their fusion and braiding matrices. 

The converse direction, as already promised, is that every object of $\Sigma^{d}\Vect_\bC$ does define a $(d+1)$-dimensional TQFT: in Theorem~\ref{thm.fullrigidity} we will prove that if $\cC$ is symmetric monoidal and \define{fully rigid} in the sense that all objects are fully dualizable, all morphisms have adjoints, etc., then so too is $\Sigma \cC$. 
Indeed, we will prove in Corollary~\ref{cor.dualizable} that, up to equivalence, $\Sigma \cC$ consists precisely of the dualizable $\cC$-modules.
The resulting tight relation between TQFTs and gapped phases of matter will, we hope, help demystify and clarify the connection between the symmetry properties of gapped systems and of the corresponding TQFTs.

\begin{remark}
  The appendix to \cite{BDSPV} collects a ``bestiary'' (their word) of $\bC$-linear $2$-categories each deserving the name ``the $2$-category of $2$-vector spaces,'' and studies dualizability criteria in each one. The main result of that appendix is that, although the subcategories of $1$-dualizable objects vary between the different $2$-categories, the subcategories of fully dualizable objects in all examples are equivalent. This common $2$-category is also equivalent to our $\Sigma \Vect_\bC$. Based on that appendix, Scheimbauer has proposed the following ``Bestiary Hypothesis'': if a $\bC$-linear symmetric monoidal $(d+1)$-category $\cC$ deserves the name ``the $(d+1)$-category of $(d+1)$-vector spaces,'' then its subcategory of fully dualizable objects is equivalent to $\Sigma^{d}\Vect_\bC$.
\end{remark}

In particular, we hope our work may help explain the following phenomenon. 
As explained in Remark~\ref{rem.framingdependence1}, 
a priori, a gapped topological phase of matter may suffer a nontrivial framing dependence. However, in practice bosonic topological orders defined over $\bC$ determine orientable TQFTs and fermionic topological orders determine spin TQFTs. Indeed,  over $\bR$, topological phases tend to determine unoriented TQFTs! 

The cobordism hypothesis description of framing dependence is as follows. The full dualizability proven in Theorem~\ref{thm.fullrigidity} implies that the space of objects of $\Sigma^d\Vect_\bK$, for any field $\bK$, or of its super analog $\Sigma^d\SVect_\bK$, carries a canonical homotopical action by the orthogonal group $\mathrm{O}(d+1)$. To define an unoriented, oriented, or spin TQFT requires finding a homotopy-fixed point for this action, for the induced $\mathrm{SO}(d+1)$-action, or for the induced $\mathrm{Spin}(d+1)$-action, respectively.  

In examples from topological phases, the homotopy-fixed point data is not always quite canonical, but seems to be as ``canonical'' as is the positive-definite metric on a vector space over $\bR$: the space of choices is not trivial, but is contractible with respect to the $\bR$-topology; 
a choice of such metric breaks the symmetry group from $\mathrm{GL}(n,\bR)$ to $\mathrm{O}(n)$, and as Lie groups these are not isomorphic but they homotopy equivalent. This is surely a symptom of the general fact that the notion of ``phase of matter'' is itself one that depends on the topology of $\bR$. Building that topology into $\Sigma^d \Vect_\bR$ has not been done. So although we expect that real=unoriented, bosonic=oriented, and fermionic=spin in a rather ``canonical'' way, we for now suggest:

\begin{conjecture}\label{conj.orientability}
  \begin{enumerate}
    \item Every object of $\Sigma^d \Vect_\bR$ carries a fixed point structure for the homotopy $\mathrm{O}(d+1)$-action.
    \item Every object of $\Sigma^d \Vect_\bC$ carries a fixed point structure for the homotopy $\mathrm{SO}(d+1)$-action.
    \item Every object of $\Sigma^d \SVect_\bC$ carries a fixed point structure for the homotopy $\mathrm{Spin}(d+1)$-action.
  \end{enumerate}
\end{conjecture}

We do not know what to expect over fields other than $\bR$ and $\bC$.
Conjecture~\ref{conj.orientability} is the strongest possible: it is easy to find examples showing that the $\mathrm{O}(d+1)$-action on $\Sigma^d \Vect_\bC$ and the $\mathrm{SO}(d+1)$-action on $\Sigma^d \SVect_\bC$ are  not trivializable. 

\begin{remark}
Conjecture~\ref{conj.orientability} might turn out to be too strong. For example, when $d=2$, we expect $\Sigma^2 \Vect_\bC$ to be equivalent to the $3$-category of multifusion categories (see Conjecture~\ref{conj.sigma2vec}), in which case, as explained in Section~3 of \cite{DSPS}, trivializing the $\mathrm{SO}(2)$- or $\mathrm{SO}(3)$-action is closely related to equipping every multifusion category with a pivotal or spherical structure, and this is a long-open conjecture \cite[Conjecture 2.8]{MR2183279}.
\end{remark}

\begin{remark}
  Our Karoubi completion $\cC \mapsto \Kar(\cC)$ is presumably closely related to the ``orbifold completion'' $\cC \mapsto \cC_{\mathrm{orb}}$ of \cite{CR12,Carqueville:2017ono,CRS17,CRS18}.
  (We did not find an analogue in that work of our Definition~\ref{defn.condensation-defect}.)
   One main difference is that 
  their construction builds in trivializations of the homotopy $\mathrm{SO}(d+1)$-action from the beginning. For example, whereas our $\Sigma \Vect_\bC = \Kar(\rB\Vect_\bC)$ has as its objects special Frobenius algebras (see Example~\ref{eg.specialfrob}), their $(\rB\Vect_\bC)_{\mathrm{orb}}$ has as its objects the \emph{symmetric} special Frobenius algebras. There is a physical reason for this difference: the authors of \cite{CR12,Carqueville:2017ono,CRS17,CRS18} are most interested in Euclidean-signature TQFTs, and so rotation-invariance is natural for them. We discovered our results by asking condensed matter questions. Microscopic condensed matter models never have continuous rotation symmetry, since they are typically built on lattices; any rotation invariance is emergent only in the low-energy. Conjecture~\ref{conj.orientability} notwithstanding, 
  there is no a priori reason for a gapped topological system to produce a true topological field theory as its low-energy effective approximation: a priori, a ``topological'' low-energy effective theory for a gapped topological system can have a nontrivial framing dependence. 
  
  Another main difference is in the packaging: we try wherever possible to hide combinatorial complexity inside inductive constructions, whereas their orbifold completion ``expands out'' such presentations, asking for example for coherent data indexed by all possible $n$-dimensional foams. It would be interesting and useful to work out the precise relationship between our constructions and theirs.
  
  Finally, while preparing this paper, we become aware of closely related work in progress by S.~Morrison and K.~Walker. We believe that their work is essentially equivalent to the orbifold completion of \cite{CR12,Carqueville:2017ono,CRS17,CRS18}, although we have not seen the details. The largest difference is that Morson and Walker work with their version of ``disk-like'' $n$-categories \cite{MR2806651}. Disk-like $n$-categories have dualizability and rotation-invariance built in, and so are particularly well-suited for Euclidean-signature TQFT, but not for condensed-matter situations with nontrivial framing dependence. 
  A side effect of whether the $n$-categories have rotation-invariance built in is that our notion of ``condensation'' from Definition~\ref{defn.condensation-defect} is not the most natural one in the Morrison--Walker setting; more natural for them is to make the additional request that the maps $f,g$ in Definition~\ref{defn.condensation-defect} be adjoint to each other in an appropriate sense. 
\end{remark}

\section{Condensations and higher Karoubi envelopes}\label{sec.Kar}

This section presents our $n$-categorical generalization of the Karoubi envelope construction. A weak $n$-category $\cC$ has objects (aka $0$-morphisms), $1$-morphisms between objects, $2$-morphisms between $1$-morphisms, etc., all the way up to $n$-morphisms, so that in particular given two objects $X,Y$, the collection of morphisms $\hom_\cC(X,Y)$ between them forms a weak $(n-1)$-category. We do not allow homotopies between $n$-morphisms --- we will not work with $(\infty,n)$-categories. That said, we expect that versions of our constructions would apply in the $(\infty,n)$-world.

What makes weak $n$-categories ``weak'' is in the composition. In general, given two composable $k$-morphisms $X \overset f \from Y \overset g \from Z$, there is not a uniquely defined compositions ``$f \circ g$,'' but rather a contractible  weak $(n-k)$-category worth of compositions (i.e.\ a weak $(n-k)$-category which is equivalent but not necessarily equal to $\{\mathrm{pt}\}$). By the same token, (higher) associativity and  unitality are imposed merely up to contractible weak higher categories of choices. The $n$-morphisms in a weak $n$-category do compose uniquely and associatively, because a ``contractible $0$-category'' is simply a one-element set.

Morally, the correct definition of \define{weak $n$-category} is that it is an $(\infty,1)$-category enriched in $(\infty,1)$-category of weak $(n-1)$-categories. The induction begins by declaring that a ``weak $0$-category'' is a set. This moral definition replaces the question of defining ``weak $n$-category'' with the question of defining ``enriched $(\infty,1)$-category'' and confirming that the $(\infty,1)$-category thereof inherits nice properties. For example, under some technical conditions on the enrichment (most notably that its selected monoidal structure is the Cartesian product), a good definition of ``enriched $(\infty,1)$-category'' is ``Segal object.'' Unpacking this provides the model of weak $n$-categories due to Tamsamani and Simpson \cite{MR1673923,MR2883823}, and for further details, we recommend the beautiful lecture notes \cite{MR3221292}. We believe that it will be straightforward to implement all of our arguments (with the exception of assumptions about colimits listed in \S\ref{subsec.abcolim}) in the Tamsamani--Simpson or any other model of weak $n$-categories, but, as mentioned already in Remark~\ref{rem:modelindependent}, we will work model-independently. 

We will henceforth generally drop the word ``weak'': all of our ``$n$-categories'' are weak $n$-categories. We will use the words ``isomorphism'' and ``equivalence'' interchangeably, and, following the principle of univalence, we sometimes use an equals sign to indicate the presence of a distinguished isomorphism. This should cause no confusion, since strict equality has no place in the world of higher categories.

\begin{example}
   $0$-categories are sets.  $1$-categories are simply ordinary categories.  $2$-categories are equivalent to bicategories: given $1$-morphisms $X \overset f \from Y \overset g \from Z$ in a (weak!)\ $2$-category, arbitrarily choose an object in the contractible $1$-category of possible compositions to be ``the'' composition; such a choice will not typically provide an associative composition of 1-morphisms, but will be associative up to an associator which will satisfy the pentagon equation.
\end{example}

\subsection{$n$-condensations}

We repeat Definition~\ref{defn.condensation-defect}, emphasizing the category number:

\begin{definition}\label{defn.condensation}
  A \define{$0$-condensation} is an equality between objects of a $0$-category. 
  Suppose that $X$ and $Y$ are objects of a weak $n$-category $\cC$.
  An \define{$n$-condensation} of $X$ onto $Y$, written $X \condense Y$, is a pair of 1-morphisms $f : X \to Y$ and $g : Y \to X$ together with an $(n-1)$-condensation  $fg \condense \id_Y$. 
  A \define{condensation} is an $n$-condensation for some $n$. In diagrams:
  \begin{equation} \tag{\ensuremath\spadesuit} \label{eqn.cond}
\begin{tikzpicture}[baseline=(midpoint)]
  \path (0,0) node (X) {$X$} (0,-2) node (Y) {$Y$} (0,-1) coordinate (midpoint);
  \draw[condensationlarge] (X) -- (Y);
\end{tikzpicture}
\quad = \quad
\begin{tikzpicture}[baseline=(midpoint)]
  \path (0,0) node (X) {$X$} (0,-.35) coordinate (Xe) (0,-2) node (Y) {$Y$} (0,-1) coordinate (midpoint);
  \draw[->] (X) .. controls +(-.75,-.75) and +(-.75,.75) ..  node[auto,swap] {$\scriptstyle f$} (Y);
  \draw[->] (Y) .. controls +(.75,.75) and +(.75,-.75) .. node[auto,swap] {$\scriptstyle g$} (X);
  \path (Y.north) +(.1,0) coordinate (Ytopa) +(-.1,0) coordinate (Ytopb);
  \draw[->] (Ytopa) .. controls +(.6,.6) and +(.5,0) .. (Xe) .. controls +(-.5,0) and +(-.6,.6) .. (Ytopb);
  \draw[doublecondensationlarge] (Xe) -- (Y);
\end{tikzpicture} \end{equation}
\end{definition}

\begin{example}
  A $1$-condensation  $X \condense Y$ is a split surjection aka retract: the $0$-condensation $fg \condense \id_Y$ is just an equation, and it forces $X \overset f \epi Y$ to be an epimorphism and $X \overset g \hookleftarrow Y$ to be a monomorphism. (This is why we have chosen the symbol ``$\condense$'': it is reminiscent of both  a hook-arrow and a double-arrow.)
  
  A $2$-condensation consists of a pair of objects $X,Y$ in a $2$-category, 1-morphisms $f : X \leftrightarrows Y : g$, and the data witnessing $\id_Y$ as a retract of $fg$. This structure does not seem to have been studied previously in the 2-categorical literature, but a special case of it arises naturally, and will be the subject of \S\ref{sec.unitalcond}. Namely, suppose that one is given objects $X,Y$ and an adjoint pair of 1-morphisms $f : X \leftrightarrows Y : g$, say with $f$ the left adjoint and $g$ the right adjoint. Recall that the data of the adjunction $f \dashv g$ consists 2-morphisms $\phi : fg \Rightarrow \id_Y$, called the \define{counit}, and $\epsilon : \id_X \Rightarrow gf$, called the \define{unit};
  these satisfy ``zig-zag equations'' that we will recall in Definition~\ref{defn.unitalcondensation} and again at the start of~\S\ref{subsec.fulldual}.
   The adjunction is called \define{separable} when $\phi$ admits a splitting $\gamma : \id_Y \to fg$ (no condition is imposed on $\epsilon$). Given a separable adjunction, choose the splitting $\gamma$ and forget the unit $\epsilon$. Then $(X,Y,f,g,\phi,\gamma)$ is a 2-condensation $X \condense Y$.
\end{example}
  
  The \define{walking $n$-condensation} (the name follows the general notion of ``walking object'' described for example in \cite{nLab-walking}) is the  $n$-category $\spadesuit_n$ freely generated by an $n$-condensation.
   In other words, an $n$-condensation in $\cC$ is precisely a diagram in $\cC$ of shape $\spadesuit_n$; for any $n$-category $\diamondsuit$, a \define{diagram of shape $\diamondsuit$} is just a functor with domain $\diamondsuit$. 
  Thus $\spadesuit_n$ is the  $n$-category generated by two objects which we will continue to called $X$ and $Y$, two $1$-morphisms $f : X \to Y$ and $g : Y \to X$, two $2$-morphisms $fg \to \id_Y$ and $\id_Y \to fg$, two $3$-morphisms, and so on, ending with two $n$-morphisms and one equation between $n$-morphisms (of the form $\alpha \beta = \id$).
Note that this list was just the list of generators, not the full list of $k$-morphisms in $\spadesuit_n$.

\begin{remark}
Higher categories like $\spadesuit_n$ freely generated by some pasting diagram are sometimes called ``computads'' in the higher category theory literature, or ``$n$-computads'' when one wants to emphasize the category number.  They will exist in any model of $n$-categories: if one cannot present an $n$-category by a pasting diagram, then one's notion of ``$n$-category'' is clearly deficient. Indeed, the ability to present $n$-categories by pasting diagrams is the fundamental axiom used in \cite{BSP2011}.

  Just like our $n$-categories, our $n$-computads are always weak. One must be a bit cautious when working with $n$-computads because of this weakness. Consider, for example, the free $3$-category with a single object $\bullet$, no nonidentity $1$-morphisms, and one generating $2$-morphism $x : \id_\bullet \to \id_\bullet$. Up to isomorphism, the complete list of $2$-morphisms in this category consists of the powers of $x$. But there are many $3$-morphisms, all invertible, parameterizing the ways that the $2$-morphism $x$ can braid around itself. For a related caution, see~\cite{MR3087175}.
\end{remark}

\begin{remark}\label{rem.truncation}
Any $(n+1)$-category $\cC$ has an \define{$n$-localization} $\lambda_{\leq n}\cC$ in which the $(n+1)$-morphisms are forced to be identities; $\lambda_{\leq n}$ is by definition the left adjoint to the inclusion $\cat{Cat}_n \mono \cat{Cat}_{n+1}$. From their presentations, it is clear that $\lambda_{\leq n}\spadesuit_{n+1} = \spadesuit_n$. (As with elsewhere in this paper, the ``$=$'' sign means that there is a canonical equivalence.) In particular, for any $n$-category $\cC$, an $n$-condensation in $\cC$ is the same as an $(n+1)$-condensation in $\cC$-thought-of-as-an-$(n+1)$-category.

Whereas one should expect that all reasonable models of ``weak $n$-category'' will be equivalent for finite $n$, there are at least two truly inequivalent versions of ``weak $\infty$-category''  \cite{134057}.
(It is unfortunate that the term ``$\infty$-category'' has come to mean something with only one dimension of noninvertibility. We will stubbornly continue to call the latter ``$(\infty,1)$-categories,'' so that ``[weak] $\infty$-category'' can mean some model-independent notion of the $n \to \infty$ limit of ``[weak] $n$-category,'' or in other words ``$(\infty,\infty)$-category.'')
 In the \define{inductive} version, a weak $\infty$-category is determined by all the ways to map into it from weak $n$-categories for finite $n$. In the \define{coinductive} version, a weak $\infty$-category is determined by all the ways to map from it to weak $n$-categories for finite $n$. One can test the difference as follows: for coinductive $\infty$-categories, infinite dualizability is the same as invertibility, since these notions become the same after truncating to any finite $n$; for inductive $\infty$-categories, infinite dualizability is weaker than invertibility.

The limit $\spadesuit := \varprojlim_n \spadesuit_n$ along the canonical maps $\spadesuit_{n+1} \to \lambda_{\leq n} \spadesuit_{n+1} = \spadesuit_n$ is most naturally interpreted as a coinductive $\infty$-category. It can also be presented inductively as a certain free $\infty$-category generated by two cells of each dimension. By construction, $\spadesuit$ has the property that if $\cC$ is an $n$-category for $n<\infty$, then the functors from $\spadesuit$ to $\cC$, i.e.\ the diagrams in $\cC$ of shape $\spadesuit$, are the same as those from $\lambda_{\leq n} \spadesuit = \spadesuit_n$. Thus, by talking in terms of $\spadesuit$, we can work with $n$-condensations for all $n$ simultaneously. We will refer to $\spadesuit$ as the \define{walking ($\infty$-)condensation}. We emphasize that this is simply a matter of linguistics: $\spadesuit$ is \emph{defined} by its maps to finite-$n$-categories, which are defined to be $n$-condensations.
\end{remark}

\begin{remark}\label{remark.commutingsquare}
Let $I = \{\bullet\to\bullet\}$ denote the walking arrow, so that the morphisms in $\cC$ are the diagrams of shape~$I$. Then a commuting square in $\cC$ is nothing but a diagram of shape $I^2$. The product of $n$-categories is computed ``levelwise'' in the following sense: the $k$-morphisms in a product are simply the product of $k$-morphisms.
(Indeed, the functor $\cC \mapsto \{\text{$k$-morphisms in }\cC\}$ is representable by the walking $k$-morphism; compare \cite[Definition 3.3]{JFS}.)
 For instance, $I$ has, up to isomorphism, a total of two $0$-morphisms and three $1$-morphisms, including identities, and so $I^2$ has four $0$-morphisms and nine $1$-morphisms, including identities and compositions.

Similarly, we define a \define{commuting condensation square} to be a diagram in $\cC$ of shape $\spadesuit^2$, and more generally a \define{commuting condensation $k$-cube} is a diagram of shape $\spadesuit^k$:
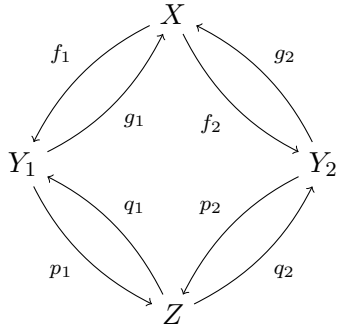
$$
\spadesuit^2 \qquad = \qquad 
\begin{tikzpicture}[scale=2, baseline=(Y1.base)]
  \path (0,0) node (X) {$X$}
     (-1,-1) node (Y1) {$Y_1$}
     (1,-1) node (Y2) {$Y_2$}
     (0,-2) node (Z) {$Z$};
  \draw[->] (X) .. controls +(-.5,-.25) and +(.25,.5) .. node[auto,swap] {$\scriptstyle f_1$} (Y1);
  \draw[->] (Y1) .. controls +(.5,.25) and +(-.25,-.5) .. node[auto,swap] (g1) {$\scriptstyle g_1$} (X);
  \draw[->] (X) .. controls +(.25,-.5) and +(-.5,.25) .. node[auto,swap] (f2) {$\scriptstyle f_2$} (Y2);
  \draw[->] (Y2) .. controls +(-.25,.5) and +(.5,-.25) .. node[auto,swap] {$\scriptstyle g_2$} (X);
  \draw[->] (Y1) .. controls +(.25,-.5) and +(-.5,.25) .. node[auto,swap] {$\scriptstyle p_1$} (Z);
  \draw[->] (Z) .. controls +(-.25,.5) and +(.5,-.25) .. node[auto,swap] (q1) {$\scriptstyle q_1$} (Y1);
  \draw[->] (Y2) .. controls +(-.5,-.25) and +(.25,.5) .. node[auto,swap] (p2) {$\scriptstyle p_2$} (Z);
  \draw[->] (Z) .. controls +(.5,.25) and +(-.25,-.5) .. node[auto,swap] {$\scriptstyle q_2$} (Y2);
  \draw[doublecondensationlarge] (X) -- (Y1); \draw[doublecondensationlarge] (X) -- (Y2); \draw[doublecondensationlarge] (Y1) -- (Z); \draw[doublecondensationlarge] (Y2) -- (Z); 
   \path (0,-1) node {$[\cdots]$};
\end{tikzpicture} 
$$

In particular, in $\spadesuit^2$, the four sides of the square $(X,Y_1,f_1,g_1,\dots)$, $(X,Y_2,f_2,g_2,\dots)$, $(Y_1,Z,p_1,q_1,\dots)$, and $(Y_2,Z,p_2,q_2,\dots)$ are all condensations, and all the subsquares commute, so that, for example, there are isomorphisms $p_1f_1 \cong p_2f_2$, $g_1q_1 \cong g_2q_2$, $f_2g_1 \cong q_2p_1$, and $f_1g_2 \cong q_1 p_2$. It is worth emphasizing that in higher categories, commutativity is data, not property: to give a commuting square, one must give these isomorphisms. This it not the end of the commutativity data: there are also commutativity isomorphisms indexed by products of higher cells in $\spadesuit^2$, indicated by the ``$[\cdots]$.''

A commuting condensation square $(X,Y_1,Y_2,Z,\dots)$  includes a diagonal condensation $X \condense Z$ given by pulling back a diagram $\spadesuit^2 \to \cC$ along the diagonal inclusion $\spadesuit \mono \spadesuit^2$.
\end{remark}

\subsection{Condensation monads}\label{subsec.condmon}

The main result underlying the ordinary Karoubi envelope construction is that a $1$-condensation $(f,g) : X \condense Y$ is uniquely determined by the composition $e := gf \in \hom(X,X)$. This map $e$ is an \define{idempotent} --- $e^2 = e$ --- and a $1$-category is \define{Karoubi complete} if all idempotents factor through (automatically unique) condensations. We must categorify the notion of idempotent.

To do so, we study the composition $e := gf$ in the walking $n$-condensation $\spadesuit_n$. One can think of $e$ either as a 1-morphism $e : X \to X$ or as an object of the monoidal $(n-1)$-category $\End_{\spadesuit_n}(X)$. Up to isomorphism, the objects of the latter are precisely the powers of $e$. Let us write $\phi : fg \Rightarrow \id_Y$ and $\gamma : \id_Y \to fg$ for the 2-morphisms that are part of the condensation $fg \condense \id_Y$. The equation $e^2 = e$ in the $1$-category case categorifies to a condensation
$$  e^2 = (gf)(gf) = g(fg)f \condense g \id_Y f = gf = e$$
in $\End_{\spadesuit_n}(X)$ with ``downward'' arrow $g.  \phi .f $ and ``upward'' arrow $g.  \gamma .f$. Note that the ``equalities'' in the above equation are really applications of the monoidality data in $\End_{\spadesuit_n}(X)$. A condensation includes morphisms in both directions, and in the case of the condensation $e^2 \condense e$, these can be thought of as a binary operation and a binary co-operation on $e$, making $e$ into a sort of (bi)algebra object.

In a monoidal $1$-category $\cD$, a binary operation $\mathrm{mult} : A^2 \to A$ on an object $A \in \cD$ is associative when the square
  $$\begin{tikzpicture}[scale = 1.5,baseline=(Y1.base)]
  \path (0,0) node (X) {$A^{3}$}
     (-1,-1) node (Y1) {$A^{2}$}
     (1,-1) node (Y2) {$A^{2}$}
     (0,-2) node (Z) {$A$};
  \draw[->] (X) -- node[auto,swap] {$\scriptstyle \mathrm{mult}.A$} (Y1);
  \draw[->] (X) -- node[auto] {$\scriptstyle A.\mathrm{mult}$} (Y2);
  \draw[->] (Y1) -- node[auto,swap] {$\scriptstyle \mathrm{mult}$} (Z);
  \draw[->] (Y2) -- node[auto] {$\scriptstyle \mathrm{mult}$} (Z);
\end{tikzpicture} $$
commutes. (As above, we have suppressed the monoidality data of $\cD$ from the diagram.) The binary operation $g.  \phi .f : e^2 \to e$ that we have identified indeed provides such a commuting square, as the two sides of
$$\begin{tikzpicture}[scale = 1.5,baseline=(Y1.base)]
  \path (0,0) node (X) {$e^{3} = gfgfgf$}
     (-1,-1) node (Y1) {$e^{2} = gfgf$}
     (1,-1) node (Y2) {$e^{2} = gfgf$}
     (0,-2) node (Z) {$e$};
  \draw[->] (X) -- node[auto,swap] {$\scriptstyle g. \phi.fgf$} (Y1);
  \draw[->] (X) -- node[auto] {$\scriptstyle gfg.  \phi.f$} (Y2);
  \draw[->] (Y1) -- node[auto,swap] {$\scriptstyle g. \phi .f$} (Z);
  \draw[->] (Y2) -- node[auto] {$\scriptstyle g. \phi.f$} (Z);
\end{tikzpicture} $$
are both equal to $g. \phi.\phi.f$.

In higher categories, commutativity of squares is data. In this case, the data involves the ``interchange'' or ``functoriality'' law for horizontal composition. Associativity of a binary operation $m : A^2 \to A$ also involves higher data, for example a pentagonator that relates various applications of associativity between the available maps $A^4 \to A$. There is a relatively easy way to parameterize these data. Specifically, for each $k \in \bN$, one must give a commuting $k$-dimensional cube with ``top'' vertex $A^{k+1}$, ``bottom'' vertex $A$, and intermediate vertices labeled by intermediate powers of $A$; let's call it the $k$-dimensional \define{associcube}. Most faces of associcubes are lower-dimensional associcubes, perhaps with some whiskering by $A$ (just as in the above $2$-dimensional cube, the top two edges are the $1$-dimensional cube $m : A^2 \to A$ with some whiskering, and the bottom two edges are unwhiskered). Some faces are, modulo associativity data of the ambient category $\cD$, simply identities. For example, one of the six faces of the three-dimensional associcube is filled in by the canonical equivalence $(A.m) \circ (m.A^2) = m.m = (m.A) \circ (A^2.m) : A^4 \to A^2$; depending on one's model for higher monoidal categories, this canonical equivalence involves various interchange/functoriality laws. These ``identity'' faces are usually collapsed, and the well-known associahedra arise as cross-sections of the collapsed associcubes.

Let us return to the endomorphism $e = gf$ in $\spadesuit_n$, and temporarily set $n=2$. In addition to the binary operation $g.  \phi .f$ there was also a co-operation $g.  \gamma .f$ coming from the upward part of $e^2 \condense e$. This co-operation is coassociative for the same reason that the binary operation was associative. Moreover, the following squares commute:
$$\begin{tikzpicture}[scale = 1.5,baseline=(Y1.base)]
  \path (0,0) node (X) {$e^{3} = gfgfgf$}
     (-1,-1) node (Y1) {$e^{2} = gfgf$}
     (1,-1) node (Y2) {$e^{2} = gfgf$}
     (0,-2) node (Z) {$e$};
  \draw[<-] (X) -- node[auto,swap] {$\scriptstyle g. \gamma.fgf$} (Y1);
  \draw[->] (X) -- node[auto] {$\scriptstyle gfg.  \phi.f$} (Y2);
  \draw[->] (Y1) -- node[auto,swap] {$\scriptstyle g. \phi .f$} (Z);
  \draw[<-] (Y2) -- node[auto] {$\scriptstyle g. \gamma.f$} (Z);
\end{tikzpicture},
\qquad
\begin{tikzpicture}[scale = 1.5,baseline=(Y1.base)]
  \path (0,0) node (X) {$e^{3} = gfgfgf$}
     (-1,-1) node (Y1) {$e^{2} = gfgf$}
     (1,-1) node (Y2) {$e^{2} = gfgf$}
     (0,-2) node (Z) {$e$};
  \draw[->] (X) -- node[auto,swap] {$\scriptstyle g. \phi.fgf$} (Y1);
  \draw[<-] (X) -- node[auto] {$\scriptstyle gfg.  \gamma.f$} (Y2);
  \draw[<-] (Y1) -- node[auto,swap] {$\scriptstyle g. \gamma .f$} (Z);
  \draw[->] (Y2) -- node[auto] {$\scriptstyle g. \phi.f$} (Z);
\end{tikzpicture} $$
These precisely say that $e$ is not just a (nonunital) associative algebra and a (noncounital) coassociative coalgebra, but in fact a \define{Frobenius algebra}. That the comultiplication splits the multiplication makes it into a \define{special} Frobenius algebra. We will go into more detail on this case in Example~\ref{eg.specialfrob}.

Together with the specialness, the four Frobenius algebra axioms --- associativity, coassociativity, and the two Frobenius identities --- can be succinctly encoded by saying that there is a commuting condensation square (in the sense of Remark~\ref{remark.commutingsquare}) with the following vertices:
$$ \begin{tikzpicture}[baseline = (Y1.base)]
  \path (0,0) node (X) {$e^{ 3}$}
     (-1,-1) node (Y1) {$e^{ 2}$}
     (1,-1) node (Y2) {$e^{ 2}$}
     (0,-2) node (Z) {$e$}
     ;
  \draw[condensation] (X) -- (Y1);
  \draw[condensation] (X) -- (Y2);
  \draw[condensation] (Y1) -- (Z);
  \draw[condensation] (Y2) -- (Z);
\end{tikzpicture} $$
The lower sides are of course the same condensation $e^2 \condense e$, and the upper sides are the two whiskerings of this condensation by $e$. Now allowing $n$ to be arbitrary, we axiomatize this structure:

\begin{definition}\label{defn.condmonad}
 A \define{condensation algebra} in a monoidal $(n-1)$-category is a sequence of commuting condensation cubes
 $$ e, \qquad  
\begin{tikzpicture}[baseline = (Y1.base)]
  \path (0,0) node (X) {$e^{ 2}$}
     (0,-1) node (Y1) {\phantom{$e$}}
     (0,-2) node (Z) {$e$}
     ;
  \draw[condensation] (X) -- (Z);
\end{tikzpicture}
  ,\qquad
\begin{tikzpicture}[baseline = (Y1.base)]
  \path (0,0) node (X) {$e^{ 3}$}
     (-1,-1) node (Y1) {$e^{ 2}$}
     (1,-1) node (Y2) {$e^{ 2}$}
     (0,-2) node (Z) {$e$}
     ;
  \draw[condensation] (X) -- (Y1);
  \draw[condensation] (X) -- (Y2);
  \draw[condensation] (Y1) -- (Z);
  \draw[condensation] (Y2) -- (Z);
\end{tikzpicture}
,\qquad
 \dots $$
 in which the faces the $k$-dimensional cube are either whiskerings of lower-dimensional cubes or are identities, following exactly the combinatorics of the associacubes recalled earlier.
 
 A \define{condensation monad} in an $n$-category $\cC$ is an object $X \in \cC$ and a condensation algebra in the monoidal $(n-1)$-category $\End_\cC(X)$. This follows the standard use of the term ``monad'' for ``algebra object in an endomorphism category.''
 
 The \define{walking condensation monad} is the $n$-category $\clubsuit_n$ such that diagrams $\clubsuit_n \to \cC$ are precisely condensation monads in $\cC$. In other words, $\clubsuit_n$ can be presented by an object $X \in \clubsuit_n$, a 1-morphism $e : X \to X$, a condensation $e^2 \condense e$, a condensation square, \dots.
\end{definition}

Note that for each $n$, $\clubsuit_n$ is finitely presentable: just like ordinary cubes, condensation cubes of dimension $\gg n$ commute as soon as their faces do. As with the case of $\spadesuit_n$ discussed in Remark~\ref{rem.truncation}, there is a canonical equivalence between the localization $\lambda_{\leq n} \clubsuit_{n+1}$ and $\clubsuit_n$, justifying omitting the category number $n$ from the names ``condensation algebra'' and ``condensation monad.'' Moreover, we may talk about the $\infty$-category $\clubsuit:=\varprojlim \clubsuit_n$. Note that, as with $\spadesuit$, this $\infty$-category is \emph{defined} by the property that for any finite $n$, for any $n$-category $\cC$, the functors $\clubsuit \to \cC$ are precisely the functors $\clubsuit_n \to \cC$.

From the construction, there is a canonical functor $\clubsuit_n \to \spadesuit_n$ taking the sole object $X \in \clubsuit_n$ to the object named $X \in \spadesuit_n$. In an earlier version of this paper, we claimed that $\clubsuit_n$ was the full subcategory of $\spadesuit_n$ on the object $X$, but our justification was insufficient. Instead, we formulate this as a conjecture:

\begin{conjecture}\label{conj.fullsuit}
  $\clubsuit$ is the full subcategory of $\spadesuit$ on the object $X$.
\end{conjecture}

\subsection{Constructing condensates}\label{subsec.uniquenessofcondensates}

Having categorified the notions of split surjection (to condensation) and idempotent (to condensation monad), we may now categorify the property that all idempotents split:

\begin{definition}\label{defn.hasallcondensates}
  The \define{condensate} of a condensation monad $(X, e, \dots) : \clubsuit_n \to \cC$, if it exists, is an extension to a condensation $(X \condense Y) : \spadesuit_n \to \cC$.
  An $n$-category $\cC$ \define{has all condensates} if all hom-$(n-1)$-categories $\hom_\cC(X,Y)$ in $\cC$ have all condensates and furthermore every condensation monad $\clubsuit_n \to \cC$ extends to a condensation $\spadesuit_n \to \cC$.
\end{definition}
To justify Definition~\ref{defn.hasallcondensates} (and in particular the use of the term ``the'' in the first sentence), we will show that condensates are unique if they exists. Note that a priori the condensates of $(X,e,\dots)$ are the objects of an $n$-category $\hom(\spadesuit_n,\cC) \times_{\hom(\clubsuit_n,\cC)} \{\mathrm{pt}\}$, where $\{\mathrm{pt}\}$ denotes the terminal $n$-category, the map $\hom(\spadesuit_n,\cC) \to \hom(\clubsuit_n,\cC)$ is restriction along the canonical inclusion, and $\{\mathrm{pt}\} \to \hom(\clubsuit_n,\cC)$ selects the condensation monad $(X,e,\dots)$.
\begin{theorem} \label{thm.uniqueness}
  Suppose that $\cC$ is an $n$-category such that all hom-$(n-1)$-categories $\hom_\cC(X,Y)$ in $\cC$ have all condensates, and let $(X,e,\dots) : \clubsuit_n \to \cC$ be a condensation monad in $\cC$. Then the $n$-category of condensates $(X,e,\dots)$ is either empty or contractible.
\end{theorem}

Our proof of Theorem~\ref{thm.uniqueness} will occupy much of this section, and will be inductive in $n$.  We briefly outline it here. We will first introduce notions of ``condensation module'' over a condensation algebra $e$ and of ``tensor product'' $\otimes_e$ of condensation modules. These notions are also the main ingredients in our higher Karoubi envelope construction, detailed in Theorem~\ref{thm.Kar}. By design, if $(X,Y,f,g,\dots)$ is a condensation, then $f$ and $g$ will be modules for the condensation algebra $e = gf$. Suppose that $(X,Y_1,f_1,g_1,\dots)$ and $(X,Y_2,f_2,g_2,\dots)$ are two condensations extending the same condensation monad $(X, e,\dots)$. We want a canonical isomorphism $Y_1 \cong Y_2$ compatible with the condensation data. We will complete the proof of Theorem~\ref{thm.uniqueness} by showing that $f_1 \otimes_e g_2: Y_2 \to Y_1$ and $f_2 \otimes_e g_1 : Y_1 \to Y_2$ provide such an isomorphism.

To arrive at the definition of ``condensation module,'' 
we build an $n$-category from two walking condensations $(X_1,Y_1,f_1,g_1,\dots)$ and $(X_2,Y_2,f_2,g_2,\dots)$, together with a morphism $\mu : Y_2 \to Y_1$. Set $m = g_1\mu f_2 : X_2 \to X_1$.
\begin{equation}\label{eqn.bimod}
\begin{tikzpicture}[baseline=(midpoint)]
  \path (0,0) node (X1) {$X_1$} +(0,-.35) coordinate (X1e) +(0,-2) node (Y1) {$Y_1$} +(0,-1) coordinate (midpoint);
  \draw[->] (X1) .. controls +(-.75,-.75) and +(-.75,.75) ..  node[auto] {$\scriptstyle f_1$} (Y1);
  \draw[->] (Y1) .. controls +(.75,.75) and +(.75,-.75) .. node[auto] {$\scriptstyle g_1$} (X1);
  \draw[doublecondensation] (X1e) -- (Y1);
  \path (3,0) node (X2) {$X_2$} +(0,-.35) coordinate (X2e) +(0,-2) node (Y2) {$Y_2$};
  \draw[->] (X2) .. controls +(-.75,-.75) and +(-.75,.75) ..  node[auto] {$\scriptstyle f_2$} (Y2);
  \draw[->] (Y2) .. controls +(.75,.75) and +(.75,-.75) .. node[auto] {$\scriptstyle g_2$} (X2);
  \draw[doublecondensation] (X2e) -- (Y2);
  \draw[->] (Y2) -- node[auto] {$\scriptstyle \mu$} (Y1);
  \path 
    (X2) +(-.5,-.25) coordinate (X2m)
    (Y2) +(-.5,.25) coordinate (Y2m)
    (X1) +(.5,-.25) coordinate (X1m)
    (Y1) +(.5,.25) coordinate (Y1m)
    ;
  \draw[->] (X2m) .. controls +(-.5,-.5) and +(-.5,.5) .. (Y2m) -- node[auto,swap] {$\scriptstyle m$} (Y1m) .. controls +(.5,.5) and +(.5,-.5) .. (X1m);
\end{tikzpicture}
\end{equation}
Let us inspect the subcategory on the objects $X_1$ and $X_2$. In addition to $m$, this category contains a condensation monad $e_1 = g_1 f_1 : X_1 \to X_1$ and a condensation monad $e_2 = g_2 f_2 : X_2 \to X_2$, and the full list of 1-morphisms consists of the powers of $e_1$, the powers of $e_2$, and the morphisms of the form $e_1^i m e_2^j$ for $i,j \geq 0$. 
There are also condensations $e_1 m \condense m$ and $m e_2 \condense m$, and more generally to identify various commuting condensation cubes relating them. 
Following the same logic that led to Definition~\ref{defn.condmonad} gives:

\begin{definition}
  Let $\cC$ be an $n$-category and $(X_1,e_1,\dots)$ and $(X_2,e_2,\dots)$ two condensation monads in $\cC$. A \define{condensation bimodule} between them consists of a $1$-morphism $m : X_2 \to X_1$ and, for each pair $(i,j) \in \bN^2$, an $(i+j)$-dimensional commuting condensation cube whose ``top'' vertex is $e_1^i m e_2^j$ and whose ``bottom'' vertex is $m$, and whose various faces are either whiskerings of lower dimensional cubes (including those that comprise the condensation monad structures on $e_1,e_2$) or identities:
  $$ m, \quad  
\begin{tikzpicture}[baseline = (Y1.base)]
  \path (0,0) node (X) {$e_1m$}
     (0,-1) node (Y1) {\phantom{$e$}}
     (0,-2) node (Z) {$e$}
     ;
  \draw[condensation] (X) -- (Z);
\end{tikzpicture}
  ,\quad
\begin{tikzpicture}[baseline = (Y1.base)]
  \path (0,0) node (X) {$me_2$}
     (0,-1) node (Y1) {\phantom{$e$}}
     (0,-2) node (Z) {$e$}
     ;
  \draw[condensation] (X) -- (Z);
\end{tikzpicture}
  ,\quad
\begin{tikzpicture}[baseline = (Y1.base)]
  \path (0,0) node (X) {$e_1^{ 2}m$}
     (-1,-1) node (Y1) {$e_1m$}
     (1,-1) node (Y2) {$e_1m$}
     (0,-2) node (Z) {$m$}
     ;
  \draw[condensation] (X) -- (Y1);
  \draw[condensation] (X) -- (Y2);
  \draw[condensation] (Y1) -- (Z);
  \draw[condensation] (Y2) -- (Z);
\end{tikzpicture}
,\quad
\begin{tikzpicture}[baseline = (Y1.base)]
  \path (0,0) node (X) {$e_1me_2$}
     (-1,-1) node (Y1) {$e_1m$}
     (1,-1) node (Y2) {$me_2$}
     (0,-2) node (Z) {$m$}
     ;
  \draw[condensation] (X) -- (Y1);
  \draw[condensation] (X) -- (Y2);
  \draw[condensation] (Y1) -- (Z);
  \draw[condensation] (Y2) -- (Z);
\end{tikzpicture}
,\quad
\begin{tikzpicture}[baseline = (Y1.base)]
  \path (0,0) node (X) {$me_2^2$}
     (-1,-1) node (Y1) {$me_2$}
     (1,-1) node (Y2) {$me_2$}
     (0,-2) node (Z) {$m$}
     ;
  \draw[condensation] (X) -- (Y1);
  \draw[condensation] (X) -- (Y2);
  \draw[condensation] (Y1) -- (Z);
  \draw[condensation] (Y2) -- (Z);
\end{tikzpicture}
,\quad
 \dots $$
 The precise combinatorics specifying which faces are whiskerings and which are identities is the same as the combinatorics that axiomatize associative bimodules in terms of commuting cubes.
\end{definition}

As with Conjecture~\ref{conj.fullsuit}, we believe but do not verify that this list of commuting condensation cubes presents the full subcategory of \eqref{eqn.bimod} on the objects $X_1,X_2$.

\begin{remark}
  Given condensation monads $(X_1,e_1,\dots)$ and $(X_2,e_2,\dots)$ in an $n$-category $\cC$, the $e_1$-$e_2$-bimodules are naturally the objects of an $(n-1)$-category whose (higher) morphisms deserve to be called \define{condensation bimodule homomorphisms}. Indeed, there is some $n$-category $\clubsuit\!\!\clubsuit_n$, the ``walking condensation bimodule,'' that parameterizes condensation bimodules, and it contains as a subcategory the disjoint union of two copies of $\clubsuit_n$; a condensation bimodule between $(X_1,e_1,\dots)$ and $(X_2,e_2,\dots)$ is an extension of the diagram $(X_1,e_1,\dots) \sqcup (X_2,e_2,\dots) : \clubsuit_n^{\sqcup 2} \to \cC$ along $\clubsuit_n^{\sqcup 2}\hookrightarrow \clubsuit\!\!\clubsuit_n$, and a condensation bimodule homomorphism is a natural transformation of functors with domain $\clubsuit\!\!\clubsuit_n$ which restricts as the identity along $\clubsuit_n^{\sqcup 2} \hookrightarrow \clubsuit\!\!\clubsuit_n$. For fixed domain, the collection of functors and natural transformations into an $n$-category $\cC$ is naturally another $n$-category. In the case at hand, it drops to an $(n-1)$-category because the inclusion $\clubsuit_n^{\sqcup 2} \hookrightarrow \clubsuit\!\!\clubsuit_n$ is essentially surjective.
\end{remark}

\begin{example}
For any object $X' \in \cC$, there is an identity condensation monad $(X', \id_{X'},\dots)$. A \define{left module} for a condensation monad $(X,e,\dots)$ is a bimodule between $(X,e,\dots)$ and $(X', \id_{X'},\dots)$ for some object $X'$. Similarly,  a \define{right module} is a bimodule between $\id_{X'}$ and $e$. One can axiomatize (left, say) $e$-modules directly by studying the full subcategory on the objects $\{X,X'\}$ of the category pictured here:
\begin{equation*} \label{eqn.leftmod}
\begin{tikzpicture}[baseline=(midpoint)]
  \path (0,0) node (X1) {$X$} +(0,-.35) coordinate (X1e) +(0,-2) node (Y1) {$Y$} +(0,-1) coordinate (midpoint);
  \draw[->] (X1) .. controls +(-.75,-.75) and +(-.75,.75) ..  node[auto] {$\scriptstyle f$} (Y1);
  \draw[->] (Y1) .. controls +(.75,.75) and +(.75,-.75) .. node[auto] {$\scriptstyle g$} (X1);
  \draw[doublecondensation] (X1e) -- (Y1);
  \path (3,-2) node  (Y2) {$X'$};
  \draw[->] (Y2) -- node[auto] {$\scriptstyle \mu$} (Y1);
  \path 
    (Y2) +(-.5,.25) coordinate (Y2m)
    (X1) +(.5,-.25) coordinate (X1m)
    (Y1) +(.5,.25) coordinate (Y1m)
    ;
  \draw[->]  (Y2m -| Y2.west) -- node[auto,swap] {$\scriptstyle m$} (Y1m) .. controls +(.5,.5) and +(.5,-.5) .. (X1m);
\end{tikzpicture}
\end{equation*}
One finds that a (left) condensation $e$-module consists of a sequence of condensation cubes $e^i m \condense m$.
\end{example}

\begin{example}
  Suppose that $n=1$, so that a condensation monad $(X,e,\dots)$ is an object $X \in \cC$ together with an idempotent $e^2 = e$. Then a  bimodule between $(X_1,e_1)$ and $(X_2,e_2)$ is a morphism $X_1 \overset m \from X_2$ such that $e_1 m = m = m e_2$.
\end{example}

\begin{example}\label{eg.canonicalmod}
  If $(X,Y,f,g,\dots)$ is a condensation, then $f$ and $g$ are respectively right and left modules for $e = gf$. A simple way to see this (for $g$) is to note that in the three-object category pictured in Example~\ref{eqn.leftmod}, composition with $\mu$ produces an isomorphism $\hom(Y,X) \cong \hom(X',X)$.
  
  Each condensation monad $e$ is itself an $e$-$e$-bimodule, simply by letting the requisite $(i,j)$th cube depend only on $i+j$.
\end{example}

\begin{example}\label{eg.ncat-condcomplete}
  Let $\{\mathrm{pt}\} \in \Cat_{n-1}$ denote the $(n-1)$-category with one object and only identity morphisms. Given a condensation monad $(\cX,E,\dots)$ in $\Cat_{n-1}$,  the $(n-1)$-category of $E$-$\{\mathrm{pt}\}$ condensation bimodules consists of functors $\cX \overset X \from \{\mathrm{pt}\}$, which is to say objects $X \in \cX$, equipped with various maps $E^k X \to E^l X$. If $E$ were a unital associative monad rather than a condensation monad, then this category would be precisely the category of what are called in the monad literature ``$E$-algebras in $\cX$.''  The functors ``forget the condensation $E$-module structure'' and ``compose with $E$'' witness this category of $E$-algebras as the condensate of $(\cX,E,\dots)$. In particular, $\Cat_{n-1}$ contains all condensates.
\end{example}

Our next task is to construct a tensor product of condensation bimodules. We will focus for clarity on the case of tensoring a right $(X,e)$-module $X'_1 \overset{m_1}\from X$ with a left $e$-module $X \overset{m_2}\from X'_2$ to produce a morphism $m_1 \otimes_e m_2 : X'_2 \to X'_1$. The bimodule case is no harder, just more notationally cumbersome.

To motivate the construction,  consider the following diagram:
\begin{equation} \label{eqn.tensor1}
\begin{tikzpicture}[baseline=(midpoint)]
  \path (0,0) node (X) {$X$} +(0,-.35) coordinate (Xe) +(0,-2) node (Y) {$Y$} +(0,-1) coordinate (midpoint);
  \draw[->] (X) .. controls +(-.75,-.75) and +(-.75,.75) ..  node[auto] {$\scriptstyle f$} (Y);
  \draw[->] (Y) .. controls +(.75,.75) and +(.75,-.75) .. node[auto] {$\scriptstyle g$} (X);
  \draw[doublecondensation] (Xe) -- (Y);
  \path (3,-2) node  (X2) {$X_2'$};
  \path (-3,-2) node  (X1) {$X_1'$};
  \draw[->] (X2) -- node[auto] {$\scriptstyle \mu_2$} (Y);
  \draw[<-] (X1) -- node[auto,swap] {$\scriptstyle \mu_1$} (Y);
  \path 
    (X2) +(-.5,.25) coordinate (X2m)
    (X1) +(.5,.25) coordinate (X1m)
    (X) +(.5,-.25) coordinate (Xm2)
    (Y) +(.5,.25) coordinate (Ym2)
    (X) +(-.5,-.25) coordinate (Xm1)
    (Y) +(-.5,.25) coordinate (Ym1)
    ;
  \draw[->]  (X2m -| X2.west) -- node[auto,swap] {$\scriptstyle m_2$} (Ym2) .. controls +(.5,.5) and +(.5,-.5) .. (Xm2);
  \draw[<-]  (X1m -| X1.east) -- node[auto] {$\scriptstyle m_1$} (Ym1) .. controls +(-.5,.5) and +(-.5,-.5) .. (Xm1);
\end{tikzpicture}
\end{equation}
There is clearly a condensation $m_1 m_2 \condense \mu_1 \mu_2$, and so there must be a corresponding condensation algebra $\epsilon \in \End(m_1m_2)$. Conversely, assuming by induction that Theorem~\ref{thm.uniqueness} holds for $(n-1)$-categories,
the morphism $\mu_1\mu_2 : X_2' \to X_1'$ is determined by the condensation monad $(m_1m_2, \epsilon,\dots)$.

\begin{proposition}\label{prop.tensorproduct}
  Suppose that $(X,e,\dots)$ is a condensation monad and that $X'_1 \overset{m_1}\from X$ and $X \overset{m_2}\from X_2'$ are right and left $e$-modules, respectively. Then $m_1m_2$ carries a canonical condensation monad~$\epsilon$. In the case pictured in (\ref{eqn.tensor1}), where $e = gf$ for a condensation $X \condense Y$ and $m_1 = \mu_1 f$ and $m_2 = g\mu_2$, the condensation monad $\epsilon \in \End( m_1m_2 )$
  has $\mu_1\mu_2 $ as its condensate.
\end{proposition}

\begin{definition}\label{defn.tensorproduct}
  Suppose that $\cC$ is an $n$-category all of whose hom $(n-1)$-categories have all condensates, and assume by induction that Theorem~\ref{thm.uniqueness} holds for $(n-1)$-categories.  With notation as in Proposition~\ref{prop.tensorproduct}, the \define{tensor product} $m_1 \otimes_e m_2 : X_2' \to X_1'$ is the condensate corresponding to the condensation monad $(m_1m_2,\epsilon,\dots)$. 
\end{definition}

In particular, in the case pictured in~(\ref{eqn.tensor1}), we have $m_1 \otimes_e m_2 = \mu_1 \mu_2$.

\begin{proof}[Proof of Proposition~\ref{prop.tensorproduct}]
 The actions of $e$ on the $m_i$s give condensations $m_1e \condense m_1$ and $em_2 \condense m_2$.
  Thinking of $e$ as an $e$-module, these condensations are examples of tensor products, providing canonical isomorphisms $m_1 \otimes_e e \cong m_1$ and $e \otimes_e m_2 \cong m_2$.
  
  By whiskering, we find two different condensations $m_1 e m_2 \condense m_1m_2$ --- in terms of tensor products, the two condensations correspond to the isomorphisms $m_1 \otimes_e em_2 \cong m_1 m_2$ and $m_1 e \otimes_e m_2 \cong m_1m_2$.
  
 Let $\epsilon_1,\epsilon_2 \in \End(m_1em_2)$ denote the corresponding condensation monads. The fact that $e$ is bimodule implies that these two condensation monads commute. Their product $\varepsilon \in \End(m_1em_2)$ is a condensation monad; the corresponding condensate, if it exists, should be thought of as ``$m_1 \otimes_e e \otimes_e m_2$.'' Commutativity further implies that we may restrict $\varepsilon$ along either condensation $m_1 e m_2 \condense m_1m_2$ to produce a condensation monad $\epsilon$ on $m_1m_2$ with condensate, if it exists, $m_1\otimes_e m_2$. But in fact the two restrictions agree, because the isomorphisms $m_1 \otimes_e e \cong m_1$ and $e \otimes_e m_2 \cong m_2$ are $e$-module isomorphisms. This $\epsilon$ is our desired condensation monad on $m_1m_2$.
\end{proof}

We may now complete the proof of Theorem~\ref{thm.uniqueness}:

\begin{proof}[Proof of Theorem~\ref{thm.uniqueness}]
  By induction, we may assume the statement of the Theorem for $(n-1)$-categories. Fix a condensation monad $(X,e,\dots) \in \cC$. Suppose that it condenses to two different condensations $(X,Y_1,f_1,g_1,\dots)$ and $(X,Y_2,f_2,g_2,\dots)$.  Then the $f_i$ and $g_i$ are $e$-modules by Example~\ref{eg.canonicalmod}, and so there are morphisms $f_1 \otimes_e g_2 : Y_2 \to Y_1$ and $f_2 \otimes_e g_1 : Y_1 \to Y_2$. Their compositions are
  \begin{gather*}
     f_1 \otimes_e g_2\, f_2 \otimes_e g_1 = f_1 \otimes_e e \otimes_e g_1 \cong f_1 \otimes_e g_1 \cong \id_{Y_1}, \\
     f_2 \otimes_e g_1\, f_1 \otimes_e g_2 = f_2 \otimes_e e \otimes_e g_2 \cong f_2 \otimes_e g_2 \cong \id_{Y_2},
  \end{gather*}
  with the last isomorphism being the special case $\mu_1 = \mu_2 = \id$ of (\ref{eqn.tensor1}).
  Thus the condensations $(X,Y_1,f_1,g_1,\dots)$ and $(X,Y_2,f_2,g_2,\dots)$ are isomorphic.

It remains to show the contractibility of the category of those endomorphisms 
$\tikz[baseline=(X1.base)]{\draw[->] node [text=white](X) {$\bullet$} (X) .. controls +(.75,.25) and +(.75,-.25) .. (X); \path (X) ++(-1.15,0) node (X1) {$(X,Y,f,g,\dots)$} ;}$
that restrict to the trivial endomorphism
$\tikz[baseline=(X1.base)]{\draw[->] node [text=white](X) {$\bullet$} (X) .. controls +(.75,.25) and +(.75,-.25) .. (X); \path (X) ++(-.75,0) node (X1) {$(X,e,\dots)$} ;}$.
 But in fact this follows from what we have already proved. Indeed, an endomorphism of a condensation $\spadesuit \to \cC$ is simply a diagram in $\cC$ of shape $\spadesuit \times \{ \tikz[baseline=(X.base)]\draw[->] node (X) {$\bullet$} (X) .. controls +(.75,.25) and +(.75,-.25) .. (X); \}$, which is the same as a condensation in the category of functors $\{ \tikz[baseline=(X.base)]\draw[->] node (X) {$\bullet$} (X) .. controls +(.75,.25) and +(.75,-.25) .. (X); \} \to \cC$.
Suppose we have a condensation $(X,Y,f,g,\dots)$ with an endomorphism which restricts trivially to $(X,e,\dots)$. Then the corresponding condensation monad in $\operatorname{Fun}(\{ \tikz[baseline=(X.base)]\draw[->] node (X) {$\bullet$} (X) .. controls +(.75,.25) and +(.75,-.25) .. (X); \} , \cC)$ is the identity morphism on the condensation monad $(X,e,\dots)$, and so condenses to the identity on $Y$. But it also condenses to the a priori nontrivial morphism of condensations, and we have already established that each condensation monad has at most one condensate.

That there are no higher morphisms of condensations can be handled similarly, by replacing the test category $\{ \tikz[baseline=(X.base)]\draw[->] node (X) {$\bullet$} (X) .. controls +(.75,.25) and +(.75,-.25) .. (X); \}$ with suitable higher categories.
\end{proof}

The constructions above  also establish:

\begin{theorem}\label{thm.Kar}
  Let $\cC$ be an $n$-category all of whose hom $(n-1)$-categories have all condensates. 
  \begin{itemize}
  \item There is an $n$-category $\Kar(\cC)$, called the \define{Karoubi envelope} of $\cC$, whose objects are the condensation monads in $\cC$, whose $1$-morphisms are the condensation bimodules, and composition of 1-morphisms is given by tensor product of condensation bimodules. The higher morphisms are the homomorphisms of condensation bimodules. 
  \item
  $\Kar(\cC)$ has all condensates.
  \item
  The functor $\cC \to \Kar(\cC)$ sending $X \mapsto (X, \id_X, \dots)$ is fully faithful, and an equivalence if $\cC$ has all condensates.
  \item
  The construction $\cC \mapsto \Kar(\cC)$ is functorial in $\cC$ and takes products of $n$-categories to products.
  \item
  The inclusion $\cC \to \Kar(\cC)$ is universal in the following sense: if $\cD$ is any $n$-category with all condensates, then any functor $\cC \to \cD$ extends uniquely (up to a contractible space of choices) to a functor $\Kar(\cC) \to \cD$.
  \end{itemize}
\end{theorem}

\begin{proof}
  It follows from Proposition~\ref{prop.tensorproduct}
  that compositions in $\Kar(\cC)$ are defined up to a contractible $(n-1)$-category worth of choices (which is the requirement for a weak $n$-category). Associativity (and higher coherences thereof) uses exactly the same argument: a sequence of bimodules $(X_0,e_0) \overset{m_1}\from (X_1,e_1) \overset{m_2}\from \dots \overset{m_k}\from (X_k,e_k)$ determines $(k-1)$-many condensation monads on $m_1m_2\cdots m_k$; that they commute is precisely the bimodule property of the $m_i$s; their unique-up-to-a-contractible-space composition is the unique-up-to-a-contractible-space condensate of the product of these commuting condensation monads. The unit bimodule on $(X,e,\dots)$ is $e$ itself as in Example~\ref{eg.canonicalmod}.
  The inclusion $\cC \to \Kar(\cC)$ sending $X \mapsto (X, \id_X, \dots)$ is a functor and fully faithful simply
  from unpacking the definition of bimodule between $(X,\id_X,\dots)$ and $(X',\id_{X'},\dots)$: such a bimodule is nothing but a morphism $X \from X'$.
  
  Suppose $((X,e,\dots), (E,\dots),\dots)$ is a condensation monad in $\Kar(\cC)$, which is to say that $(E,\dots)$ is a bimodule between $(X,e,\dots)$ and itself, equipped with the data making it into a condensation monad. Then $E : X \to X$ is a condensation monad on $X$, and the extra data is precisely the necessary data needed to make $e$ and $E$ commute. The condensate in $\Kar(\cC)$ of $((X,e,\dots), (E,\dots),\dots)$ is the product condensation monad $(X, eE, \dots)$ in $\cC$.
  If $\cC$ already had all condensates, then the assignment sending each condensation monad to its condensate, which is well-defined (up to contractible choices) by Theorem~\ref{thm.uniqueness}, provides the inverse to the inclusion $\cC \to \Kar(\cC)$.
  
  The second-to-last claim in the Theorem follows from the facts that: $\Kar(\cC)$ is built entirely out of diagrams in $\cC$ (as opposed to, say, colimits in $\cC$); diagrams are preserved by arbitrary functors; for 
  any diagram shape $\diamondsuit$, a $\diamondsuit$-shaped diagram 
  in $\cC_1 \times \cC_2$ is the same as a $\diamondsuit$-shaped diagram in $\cC_1$ together with a $\diamondsuit$-shaped diagram in $\cC_2$.
  
  Given a functor $\cC \to \cD$, functoriality of $\Kar(-)$ extends it to $\Kar(\cC) \to \Kar(\cD)$. But if $\cD$ has all condensates, then $\Kar(\cD) \cong \cD$. This gives the existence part of the last claim in the Theorem. The uniqueness follows from the fact that condensation monads, and condensations, are preserved by all functors, and that condensates are unique when they exist.
\end{proof}

\begin{remark}
If the hom $(n-1)$-categories in $\cC$ did not have all their condensates, then to construct $\Kar(\cC)$ we would first take the Karoubi envelope of all the hom $(n-1)$-categories; without this step, the composition of bimodules might not exist or might not by well-defined up to a contractible space. This local Karoubi completion step is allowed
(in the sense that taking the Karoubi envelope of all hom-$(n-1)$-categories in an $n$-category produces a new $n$-category)
 because of the second-to-last bullet point of Theorem~\ref{thm.Kar}. By induction, for arbitrary $\cC$ with no conditions on existence of condensates, $\Kar(\cC)$ will again by built entirely out of diagrams in $\cC$, but those diagrams will be more complicated. For example, 
if the hom $(n-1)$-categories in $\cC$ do not have all  condensates, then
among the objects of $\Kar(\cC)$ are ``condensation monads'' $(X,e,\dots)$ in $\cC$ where the ``morphism'' $e$ is not just a 1-morphism in $\cC$ but is itself a condensation monad $e = (e_0,\epsilon,\dots)$ in $\End(X)$.
\end{remark}

\begin{example}\label{example.sigmadefn}
  As explained in \S\ref{subsec.intro-Kar}, the application of Theorem~\ref{thm.Kar} that we will be most interested in is the following. Suppose that $\cC$ is a monoidal $(n-1)$-category with all condensates. Write $\rB \cC$ for the corresponding one-object monoidal $n$-category. Then $\rB \cC$ satisfies the conditions of Theorem~\ref{thm.Kar}, and so there is a universal $n$-category $\Kar(\rB \cC)$ with all condensates and a fully faithful map $\rB \cC \to \Kar(\rB \cC)$. The importance of the construction $\Kar(\rB(-))$ justifies giving it a simple name. We will call it $\Sigma(-)$.
  
  The reason for the name ``$\Sigma(-)$'' is because we think of it as a type of ``suspension'' operation. The construction $\rB : \{\text{monoidal $(n-1)$-categories}\} \to \{\text{pointed $n$-categories}\}$ is left-adjoint to the \define{loops} construction that takes a pointed $n$-category $(\cC, 1 \in \cC)$ to $\Omega\cC := \End_\cC(1)$. (An $n$-category is \define{pointed} when it is equipped with a distinguished object.) The universality in the last part of Theorem~\ref{thm.Kar} says that $\Kar(-)$ is the left-adjoint of the inclusion $\{\text{$n$-categories with all condensates}\} \hookrightarrow \{\text{$n$-categories}\}$. Together with the full faithfulness of $\rB\cC \hookrightarrow \Kar(\rB\cC)$, we find that $\Sigma(-)$ is the left-adjoint to $\Omega(-)$ in the world of condensation-complete categories.
  
  Suppose that $\cC$ is not just monoidal but symmetric monoidal. Then $\rB \cC$ is also symmetric monoidal (compare \cite[\S3.3]{MR3924174}) and, since $\Kar(-)$ preserves products, so is $\Kar(\rB \cC) = \Sigma\cC$. This allows us to iterate suspension, defining for example a symmetric monoidal $(d+1)$-category $\Sigma^d \Vect_\bK$ for any field $\bK$.
\end{example}

\subsection{Boundary conditions and lattice Hamiltonians}\label{sec.bc}

When can a gapped topological system $Y$ be reconstructed from a gapped topological system $X$ by ``condensing'' some degrees of freedom?
Let us list some reasonable hypotheses about physical condensation of gapped topological systems. By the end of the list, we will discover the category-theoretical condensations and condensation monads described above.

First, let us require that a physical condensation of $X$ onto $Y$ can be performed in one macroscopic region, independently of  other distant macroscopic regions. Specifically, let us pick a coordinate direction $x$, and suppose that it is possible to condense to phase $Y$ in the domain $x\ll 0$ while remaining in system $X$ for $x \gg 0$. After zooming out to a sufficiently macrosopic length scale, such a partial condensation will produce a codimension-$1$ interface $f$ from $X$ to $Y$. (In order to match composition order for functions, our convention will be that an interface ``from $X$ to $Y$'' is one with $X$ on the right and $Y$ on the left.)
If instead we condensed to phase $Y$ in the $x \gg 0$ region while remaining in phase $X$ for $x \ll 0$, then we would produce an interface $g$ from $Y$ to $X$. We require that whatever this ``partial condensation'' procedure is, the interfaces $f$ and $g$ it produces should be gapped and topological. In particular, they may be translated freely in the $x$-direction without changing their low-energy properties, so long as they remain sufficiently far from any other defect in the system.

Now, for some intermediate length scale $L>0$, consider condensing from $X$ to $Y$ both for $x<0$ and for $x>L$. We should be able to complete the job to obtain a uniform $Y$-phase if we also condense in the intermediate region $0 \leq x \leq L$. The two partial condensations together produce the composition $fg$ of interfaces from $Y$ to $X$ and back to $Y$. Being able to complete the job is simply the ability to condense from the $fg$ interface to the invisible interface $\id_Y$. By induction, we learn:

\begin{lemma}
  Under mild assumptions about the nature of ``physical condensation'' of gapped topological systems, any physical condensation of phase $X$ onto phase $Y$ produces a condensation $X \condense Y$ in the sense of Definitions~\ref{defn.condensation-defect} and~\ref{defn.condensation}. \qed
\end{lemma}

In order to fully identify physical condensation with Definitions~\ref{defn.condensation-defect} and~\ref{defn.condensation}, we must reverse the construction, explaining how to produce phase $Y$ by condensing some degrees of freedom in phase $X$ if we are given the interfaces $f$ and $g$ and a way of condensing $fg$ onto $\id_Y$.
We will employ a standard condensed matter procedure:  we will enlarge the 
low-energy system of $X$ by proliferating some collection of defects and then adding some new local Hamiltonian on these 
extra degrees of freedom which makes the new system $Y$ gapped and topological. 
 Our method will be a physical version of the categorical discussion from \S\ref{subsec.uniquenessofcondensates} of producing the condensate of a given condensation monad.

To this end, consider the opposite composition $e = gf$. It represents a thin slab of the $Y$ phase embedded within the $X$ phase.
Having started with a categorical condensation $X \condense Y$, this $e$-interface carries a condensation monad structure. 
By induction,
the condensation $e^2 \condense e$
  is precisely the data needed to fill in the space between two $Y$ slabs to produce a single thicker slab, and 
  associativity produces the information how to partially merge collections of consecutive slabs of $Y$ within a bulk of $X$ phase. 

Crucially, 
merely the condensation monad $e$ contains enough information to give a sort of mesoscopic ``topological effective description'' of the original microscopic 
condensation procedure. To obtain $Y$:
\begin{enumerate}
  \item Proliferate the $e$ defects throughout $X$. Imagine these defects as thin slabs of $Y$.
  \item Turn on local Hamiltonians which merge consecutive slabs into a uniform $Y$ phase. \label{step2}
\end{enumerate}
Why can step (\ref{step2}) be performed? For any pair of consecutive slabs, the condensation $e^2 \condense e$ is a condensation of $(d-1)$-dimensional systems, and so by induction has a local Hamiltonian description. But the higher commuting condensation cubes in Definition~\ref{defn.condmonad} means that these local Hamiltonians commute!

\begin{remark}
More precisely, the higher-dimensional commuting condensation cubes provide us with compatible collections of invertible maps which 
can be used to deform the neighbouring networks of defects away from the location where any given local projector will act. The actual projectors are 
defined by conjugating the original local projectors by the maps which move away other spectating defects.
\end{remark}

The induction starts with $(0+1)$-dimensional, aka anyon, condensation, where it produces a projector, aka idempotent, Hamiltonian (compare \S\ref{intro.layering} and Example~\ref{eg.anyoncondensation}).
After unpacking the recursion, we find a network of interfaces and defects in an ambient $X$ phase equipped with a (complicated, but algorithmic) commuting projector Hamiltonian. This network gives a mesoscopic ``topological effective description'' of the phase $Y$. 

The uniqueness proved in Theorem~\ref{thm.uniqueness} is exactly what is needed to show that, if we did start with some generic ``physical condensation'' of $X$ onto $Y$ and then ran the above procedure to first extract a categorical condensation $X \condense Y$ and then to build a mesoscopic lattice, then we would produce a system in the phase $Y$ (and not some other phase). Specifically, Theorem~\ref{thm.uniqueness} produces a canonical invertible defect between the original system $Y$ and the mesoscopic lattice phase.

\begin{remark}\label{remark.physicalbimodule}
  Condensation bimodules can be interpreted similarly.
  Suppose that we have two systems $X_1$ and $X_2$ which can be condensed to $Y_1$ and $Y_2$ respectively, and 
we want to describe some interface $\mu$ interpolating from $Y_1$ to $Y_2$. We get for free an interface $m = g_2 \mu f_1$ from $X_1$ to $X_2$, together with 
condensations which allow one to merge the $m$ interface with slabs $e_1$ and $e_2$ of $Y_1$ or $Y_2$. This data will allow us to reconstruct $\mu$ by a mesoscopic  
condensation procedure, where we proliferate $e_1$ and $e_2$ interfaces on the two sides of $m$ and then turning on local (commuting projector) Hamiltonians 
which merge the consecutive $Y_1$ and $Y_2$ slabs to each other and to $m$. As long as we are given an $m$ with the correct properties, the procedure will 
always go through and give some gapped topological defect $\mu$.
\end{remark}

In the special case when $X$ is in the trivial phase, so that $f$ and $g$ are boundaries for $Y$ and $e$ is actually a phase in one dimension lower, we find a construction of $Y$ from a network of $e$-systems, analogous to the string-net construction of a quantum double model 
in $(2+1)$ dimensions \cite{Levin:2004mi}.  Recall from \S\ref{subsec.intro-Kar} 
and Example~\ref{example.sigmadefn}
the notation $\Sigma\cC = \Kar(\rB \cC)$, where $\cC$ is a symmetric monoidal $n$-category with all condensates, and $\rB\cC$ is the one-object $(n+1)$-category with endomorphism category $\cC$. Note that $\rB\cC$ is symmetric monoidal if $\cC$ is --- see \cite[\S3.3]{MR3924174} --- and hence so too is $\Sigma\cC$ by employing the last sentence of Theorem~\ref{thm.Kar}.
By induction, we find:

\begin{theorem} \label{thm.Hamiltonian}
  Let $\cV$ denote the symmetric monoidal $1$-category of gapped topological $(0+1)$-dimensional systems. For example, for bosonic systems without any symmetry or time reversal enhancement, $\cV$ is equivalent to the category $\Vect_\bC$ of finite-dimensional vector spaces over $\bC$, for fermionic systems $\cV$ is equivalent to the category $\SVect_\bC$ of finite-dimensional supervector spaces, and for $G$-enriched bosonic phases $\cV$ is equivalent to the category $\cat{Rep}(G)$ of finite-dimensional $G$-representations.
  
  Then $\Sigma^d\cV$ is equivalent to the category of $(d+1)$-dimensional gapped topological phases which can be condensed, via gapped topological interfaces, from the vacuum. Furthermore, every object in $\Sigma^d\cV$ (and every morphism, by Remark~\ref{remark.physicalbimodule}) determines a commuting projector Hamiltonian system. That commuting projector Hamiltonian system provides a mesoscopic ``topological effective description'' of the corresponding $(d+1)$-dimensional phase. \qed
\end{theorem}

By ``$(d+1)$-dimensional'' we mean, of course, a system in $d$ spatial dimensions plus $1$ time dimension.
Corollary~\ref{cor.tqft} gives a TQFT analogue.

\begin{example}\label{eg.anyoncondensation}
  Suppose $X$ and $Y$ are $(0+1)$-dimensional objects, aka anyons, possibly embedded into a larger ambient gapped phase $\Phi$. The ``interfaces'' $f$ and $g$ are local operations which can interpolate between the two anyons. 
The condition $fg = \id_Y$ tells us that the composition of the two local operations acts trivially on the anyon $Y$, 
i.e.\ $f$ and $g$ tell us how to embed the anyon $Y$ as a direct summand within $X$. 

  If we are given any Hamiltonian realization of $\Phi$ with a local defect trapping the $X$ anyon at some location, 
then we can simply add a small negative multiple of the projector $e = g f$ to the Hamiltonian in order to get a new Hamiltonian 
whose ground states only include the anyon $Y$ at that location. This is the condensation procedure in $(0+1)$ dimensions.
\end{example}

\begin{example}
In $(1+1)$ dimensions, $X$ and $Y$ will be one-dimensional string defects, possibly embedded into a larger ambient gapped phase $\Phi$. 
The $f$ and $g$ interfaces will be anyons interpolating between $X$ and $Y$ and we require the fusion $f g$ to contain the trivial anyon $\id_Y$ 
as a direct summand. On the other hand, $e$ will be some anyon in $X$, such that we can find $e$ as a direct summand in $e^2$, 
with appropriate associativity constraints identifying a canonical choice of the $e$ summand inside $e^3$ (or any higher power of $e$).
In fact, 
as explained more fully in Example~\ref{eg.specialfrob}, $e$ will be a (nonunital) special Frobenius algebra object in the category of anyons in $X$.

If we have some Hamiltonian realization of $\Phi$, $X$, and $e$, then we can introduce a  mesoscopic lattice of $e$ defects. 
We can produce a new Hamiltonian by adding to the old one some small negative multiple of the projectors 
$e^2 \to e$ for all consecutive pairs of $e$ anyons. This gives an effective Hamiltonian description of the $Y$ system
as the condensate of $e$-anyons in $X$.

This $(1+1)$-dimensional anyon condensation procedure is well known. See e.g. \cite{Kapustin:2010if,CR12}
\end{example}

\begin{example}
  In $(2+1)$ dimensions, $f$ and $g$ will be string defects, with $fg$ being a string which can  condense to the trivial interface $\id_Y$ 
by condensing some anyon supported on the string $fg$. 
The interface $e$ will be some interface in $X$ with the property that the composition $e^2$ can be condensed to $e$ 
by condensing some anyon supported on $e^2$. This anyon can be thought of as a defect at which $e^2$ is fused to $e$ and then back to $e^2$.
Again, there will be various associativity constraints which guarantee that consecutive pairs of interfaces in a long composition $e^n$ 
can be condensed independently of each other. These constraints give invertible operators which interpolate between configurations 
decorated by networks of $e$ defects of various topology and projectors which can be used to eliminate ``bubbles'' from the network.

The construction of an Hamiltonian for $Y$ proceeds as above, starting from a Hamiltonian for $X$ decorated by some 
regular network of $e$ defects and adding small negative multiples of the projectors which eliminate individual faces of the 
network. A version of this construction has been studied in 
\cite{Carqueville:2017ono,CRS18}.
\end{example}

\section{Unitality}

This section is motivated by a mismatch between between our condensation algebras and bimodules, and the algebras and bimodules usually used in topological field theory. Our goal is to show that under mild hypotheses, the two versions  give equivalent results. The mismatch is visible already in two (spacetime) dimensions. (Although our motivation is in category number $n=2$, we will of course phrase as many results as possible for arbitrary $n$.) Two-dimensional TQFTs are usually described in terms of (finite-dimensional) unital associative algebras which enjoy the property of being separable \cite{Schommer-Pries:thesis}. 
As explained in \S\ref{subsec.condmon}, the condensation algebras that make up our 2-category $\Sigma\Vect = \Kar(\rB\Vect)$ are, instead, a type of nonunital noncounital bialgebra. In detail, we have:

\begin{example}\label{eg.specialfrob}
  Let $(\cC,\otimes)$ be a monoidal $1$-category. A condensation algebra in $\cC$, i.e.\ a condensation monad in the one-object $2$-category $\rB\cC$, consists of the following data:
  \begin{itemize}
    \item An object $e \in \cC$.
    \item A ``multiplication'' map $\mathrm{mult}: e \otimes e \to e$.
    \item A ``comultiplication'' map $\mathrm{comult}: e \to e \otimes e$.
  \end{itemize}
  These data are subject to the following axioms:
  \begin{enumerate}
    \item The multiplication and comultiplication are together a split surjection $e^{\otimes 2} \condense e$. \label{item.special}
    \item The multiplication map $e \otimes e \to e$ is associative.   
    \label{item.assoc}
    \item The comultiplication map $e \to e \otimes e$ is coassociative.  
    \label{item.coassoc}
    \item The idempotent $\mathrm{comult} \circ \mathrm{mult} : e^{\otimes 2} \to e \to e^{\otimes 2}$ is equal to both  $(e \otimes \mathrm{mult}) \circ (\mathrm{comult} \otimes e) : e^{\otimes 2} \to e^{\otimes 3} \to e^{\otimes 2}$ and $(\mathrm{mult} \otimes e) \circ (e \otimes \mathrm{comult}) : e^{\otimes 2} \to e^{\otimes 3} \to e^{\otimes 2}$. 
    \label{item.frob}
  \end{enumerate}
  These axioms are precisely the axioms for $e$ to be a \define{nonunital special Frobenius algebra}. Axioms (\ref{item.assoc}), (\ref{item.coassoc}), and (\ref{item.frob}) define ``nonunital Frobenius algebra'' (axiom (\ref{item.frob}) is called the \define{Frobenius axiom}), and axiom (\ref{item.special}) is the definition of ``special.''
  There is a familiar string diagrammatics for bicategories and monoidal categories. Our convention will be to read composition from left to right and from bottom to top. Writing $e = \; 
  \begin{tikzpicture}[baseline=(basepoint)]
    \path(0,2pt) coordinate (basepoint);
    \draw[onearrow] (0,-1pt) -- (0,12pt);
  \end{tikzpicture}\;$ and the multiplication and comultiplication as $\; 
  \begin{tikzpicture}[baseline=(basepoint)]
    \path(0,2pt) coordinate (basepoint);
    \draw[onearrow] (-4pt,-1pt) -- (0,6pt); \draw[onearrow] (4pt,-1pt) -- (0,6pt); \draw[onearrow](0,6pt) -- (0,12pt);
  \end{tikzpicture}\;$ and $
  \; 
  \begin{tikzpicture}[baseline=(basepoint)]
    \path(0,2pt) coordinate (basepoint);
    \draw[onearrow] (0,-1pt) -- (0,5pt); \draw[onearrow] (0,5pt) -- (-4pt,12pt); \draw[onearrow] (0,5pt) -- (4pt,12pt);
  \end{tikzpicture}\;$, the axioms read:
  $$
  \begin{tikzpicture}[baseline=(basepoint)]
    \path (0,1.4) coordinate (basepoint);
    \draw[onearrow] (0,.5) -- (0,1);
    \draw[onearrow] (0,1) .. controls +(.25,.5) and +(.25,-.5) .. (0,2);
    \draw[onearrow] (0,1) .. controls +(-.25,.5) and +(-.25,-.5) .. (0,2);
    \draw[onearrow] (0,2) -- (0,2.5);
  \end{tikzpicture}
  \;=\;
  \begin{tikzpicture}[baseline=(basepoint)]
    \path (0,1.4) coordinate (basepoint);
    \draw[onearrow] (0,.5) -- (0,2.5);
  \end{tikzpicture}
  \;\;, \quad
  \begin{tikzpicture}[baseline=(basepoint),xscale=.75]
    \path (0,1.4) coordinate (basepoint);
    \draw[onearrow] (0,.5) -- (.5,1);
    \draw[onearrow] (1,.5) -- (.5,1);
    \draw[onearrow] (.5,1) -- (1,2);
    \draw[onearrow] (2,.5) -- (1,2);
    \draw[onearrow] (1,2) -- (1,2.5);
  \end{tikzpicture}
   = 
  \begin{tikzpicture}[baseline=(basepoint),xscale=-.75]
    \path (0,1.4) coordinate (basepoint);
    \draw[onearrow] (0,.5) -- (.5,1);
    \draw[onearrow] (1,.5) -- (.5,1);
    \draw[onearrow] (.5,1) -- (1,2);
    \draw[onearrow] (2,.5) -- (1,2);
    \draw[onearrow] (1,2) -- (1,2.5);
  \end{tikzpicture}
  , \quad
  \begin{tikzpicture}[baseline=(basepoint),xscale=.75]
    \path (0,1.4) coordinate (basepoint);
    \draw[onearrow] (1,.5) -- (1,1);
    \draw[onearrow] (1,1) -- (.5,2);
    \draw[onearrow] (.5,2) -- (0,2.5);
    \draw[onearrow] (.5,2) -- (1,2.5);
    \draw[onearrow] (1,1) -- (2,2.5);
  \end{tikzpicture}
   = 
  \begin{tikzpicture}[baseline=(basepoint),xscale=-.75]
    \path (0,1.4) coordinate (basepoint);
    \draw[onearrow] (1,.5) -- (1,1);
    \draw[onearrow] (1,1) -- (.5,2);
    \draw[onearrow] (.5,2) -- (0,2.5);
    \draw[onearrow] (.5,2) -- (1,2.5);
    \draw[onearrow] (1,1) -- (2,2.5);
  \end{tikzpicture}
  , \quad
  \begin{tikzpicture}[baseline=(basepoint)]
   \path (0,.9) coordinate (basepoint);
   \draw[onearrow] (0,0) -- (.5,.75); \draw[onearrow] (1,0) -- (.5,.75); 
   \draw[onearrow] (.5,.75) -- (.5,1.25);
   \draw[onearrow] (.5,1.25) -- (0,2); \draw[onearrow] (.5,1.25) -- (1,2); 
  \end{tikzpicture}
  = \;
  \begin{tikzpicture}[baseline=(basepoint)]
   \path (0,.9) coordinate (basepoint);
   \draw[onearrow] (0,0) -- (0,.5);
   \draw[onearrow] (0,.5) -- (0,2);
   \draw[onearrow] (0,.5) -- (1,1.5);
   \draw[onearrow] (1,0) -- (1,1.5);
   \draw[onearrow] (1,1.5) -- (1,2);
  \end{tikzpicture}
  \; = \;
  \begin{tikzpicture}[baseline=(basepoint),xscale=-1]
   \path (0,.9) coordinate (basepoint);
   \draw[onearrow] (0,0) -- (0,.5);
   \draw[onearrow] (0,.5) -- (0,2);
   \draw[onearrow] (0,.5) -- (1,1.5);
   \draw[onearrow] (1,0) -- (1,1.5);
   \draw[onearrow] (1,1.5) -- (1,2);
  \end{tikzpicture}
  \;.
  $$
  \begin{remark}
  Axioms (\ref{item.assoc}) and (\ref{item.coassoc}) follow from axioms (\ref{item.special}) and (\ref{item.frob}).  Checking this is a fun exercise left to the reader. Hint: Create a bubble using (\ref{item.special}), move it around using (\ref{item.frob}), and then collapse it using~(\ref{item.special}).\end{remark}
  
  A left $e$-module $m$, which we will denote in our string diagrams simply as a solid edge $\; 
  \begin{tikzpicture}[baseline=(basepoint)]
    \path(0,2pt) coordinate (basepoint);
    \draw[ultra thick] (0,-1pt) -- (0,12pt);
  \end{tikzpicture}\;$, comes with a left $e$-action $\; 
  \begin{tikzpicture}[baseline=(basepoint)]
    \path(0,2pt) coordinate (basepoint);
    \draw[onearrow] (-4pt,-1pt) -- (4pt,7pt); \draw[ultra thick] (4pt,-1pt) -- (4pt,12pt);
  \end{tikzpicture}\; : e \otimes m \to m$ and a left $e$-coaction $ \; 
  \begin{tikzpicture}[baseline=(basepoint)]
    \path(0,2pt) coordinate (basepoint);
    \draw[onearrow] (4pt,4pt) -- (-4pt,12pt); \draw[ultra thick] (4pt,-1pt) -- (4pt,12pt);
  \end{tikzpicture}\; :
  m \to e \otimes m$, satifying versions of specialness, (co)associativity, and the Frobenius relation:
  $$ 
  
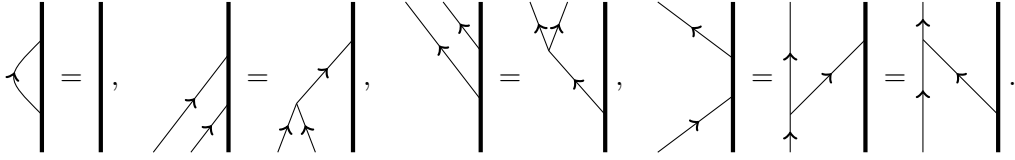
\begin{tikzpicture}[baseline=(basepoint)]
    \path (0,.9) coordinate (basepoint);
    \draw[ultra thick] (0,0) -- (0,2);
    \draw[onearrow] (0,.5) .. controls +(-.5,.5) and +(-.5,-.5) .. (0,1.5);
  \end{tikzpicture}
  \;=\;
  \begin{tikzpicture}[baseline=(basepoint)]
    \path (0,.9) coordinate (basepoint);
    \draw[ultra thick] (0,0) -- (0,2);
  \end{tikzpicture}
  \;, \quad
  \begin{tikzpicture}[baseline=(basepoint)]
    \path (0,.9) coordinate (basepoint);
    \draw[ultra thick] (0,0) -- (0,2);
    \draw[onearrow] (-.5,0) -- (0,.65);
    \draw[onearrow] (-1,0) -- (0,1.3);
  \end{tikzpicture}
  \; = 
  \begin{tikzpicture}[baseline=(basepoint)]
    \path (0,.9) coordinate (basepoint);
    \draw[ultra thick] (0,0) -- (0,2);
    \draw[onearrow] (-.5,0) -- (-.75,.65);
    \draw[onearrow] (-1,0) -- (-.75,.65);
    \draw[onearrow] (-.75,.65) -- (0,1.5);
  \end{tikzpicture}
  \;, \quad
  \begin{tikzpicture}[baseline=(basepoint)]
    \path (0,.9) coordinate (basepoint);
    \draw[ultra thick] (0,0) -- (0,2);
    \draw[onearrow] (0,.7) -- (-1,2);
    \draw[onearrow] (0,1.35) -- (-.5,2);
  \end{tikzpicture}
  \; = 
  \begin{tikzpicture}[baseline=(basepoint)]
    \path (0,.9) coordinate (basepoint);
    \draw[ultra thick] (0,0) -- (0,2);
    \draw[onearrow] (0,.5) -- (-.75, 1.35);
    \draw[onearrow] (-.75,1.35) -- (-.5,2);
    \draw[onearrow] (-.75,1.35) -- (-1,2);
  \end{tikzpicture}
  \;, \quad
  \begin{tikzpicture}[baseline=(basepoint)]
   \path (0,.9) coordinate (basepoint);
   \draw[onearrow] (0,0) -- (1,.75);  
   \draw[onearrow] (1,1.25) -- (0,2); 
   \draw[ultra thick] (1,0) -- (1,2);
  \end{tikzpicture}
  \; = \;
  \begin{tikzpicture}[baseline=(basepoint)]
   \path (0,.9) coordinate (basepoint);
   \draw[onearrow] (0,0) -- (0,.5);
   \draw[onearrow] (0,.5) -- (0,2);
   \draw[onearrow] (0,.5) -- (1,1.5);
   \draw[ultra thick] (1,0) -- (1,2);
  \end{tikzpicture}
  \; = \;
  \begin{tikzpicture}[baseline=(basepoint),xscale=-1]
   \path (0,.9) coordinate (basepoint);
   \draw[ultra thick] (0,0) -- (0,2);
   \draw[onearrow] (0,.5) -- (1,1.5);
   \draw[onearrow] (1,0) -- (1,1.5);
   \draw[onearrow] (1,1.5) -- (1,2);
  \end{tikzpicture}
  \;.
  $$
\end{example}

\begin{remark}\label{remark.special}
  In a monoidal $n$-category with $n>1$, a condensation algebra is again a {special} Frobenius algebra, defined so as to accommodate the higher categoricity:  first, (co)associativity and the Frobenius axioms are imposed only up to coherent homotopy (all the coherence data takes the shape of various associahedra); second, the specialness axiom is modified to the presence of a condensation $\mathrm{mult} \circ \mathrm{comult} \condense \id_e$. The data of this condensation provide the ingredients needed to draw bubbles of dimension $\leq n$, and say that an $n$-dimensional bubble can be created or removed. The (co)associativity and Frobenius axioms allow such a bubble to be moved around freely.
\end{remark}

\subsection{Unital condensations}\label{sec.unitalcond}

Consider a condensation $X \condense Y$ in an $n$-category $\cC$ with $n\geq 2$. The data of this condensation includes: 
\begin{itemize}
\item the objects $X$ and $Y$; 
\item morphisms $f : X \to Y$ and $g : Y \to X$;
\item $2$-morphisms $\phi : fg \Rightarrow \id_Y$ and $\gamma : \id_Y \Rightarrow fg$;
\item $3$- through $n$-morphisms. When $n=2$, these consist just of the equality $\phi\gamma = \id_{\id_Y}$.
\end{itemize}

\begin{definition}\label{defn.unitalcondensation}
  With notation as above, the condensation $X \condense Y$ is \define{unital} if $\phi : fg \Rightarrow \id_Y$ is the counit of an adjunction $f \dashv g$, i.e.\ if there exists a \define{unit}  $\eta : \id_X \Rightarrow e= gf$ such that the ``zig-zag'' compositions
  \begin{gather*}
   f \overset{f \eta}\Longrightarrow fgf \overset{\phi f}\Longrightarrow f, \\
   g \overset{\eta g}\Longrightarrow gfg \overset{g\phi}\Longrightarrow g
  \end{gather*}
  are (equivalent to) identities.
  
  The condensation $X \condense Y$ is \define{counital} if $\gamma : \id_Y \Rightarrow fg$ is the unit of an adjunction $g \dashv f$, i.e.\ if there exists a \define{counit}  $\epsilon : e=gf \Rightarrow \id_X$ such that the ``zig-zag'' compositions
  \begin{gather*}
   f \overset{\gamma f}\Longrightarrow fgf \overset{f \epsilon}\Longrightarrow f,\\
   g \overset{g \gamma}\Longrightarrow gfg \overset{\epsilon f}\Longrightarrow g
  \end{gather*}
  are (equivalent to) identities.
\end{definition}

  We will repeat and elaborate on this (well-known) definition of \define{adjunction} $f \dashv g$ between 1-morphisms in a higher category in \S\ref{subsec.fulldual}. In the special case of $2$-categories, the paper
\cite{Reutter2018}
proposes categorifying the $1$-categorical notion of ``split surjection'' not to our notion of ``condensation'' but rather to the following notion of ``separable adjunction'':
\begin{definition}\label{defn.sepadj}
A \define{separable adjunction} in a $2$-category consists of objects $X$ and $Y$, $1$-morphisms $f : X \to Y$ and $g : Y \to X$, and $2$-morphisms $\phi : fg \Rightarrow \id_Y$ and $\eta : \id_X \Rightarrow gf$ satisfying the first set of conditions from Definition~\ref{defn.unitalcondensation}, such that there exists a $2$-morphism $\gamma : \id_Y \Rightarrow fg$ with $\phi\gamma = \id_{\id_Y}$. More generally, an adjunction $f \dashv g$ in an $n$-category $\cC$ is \define{separable} if the counit of the adjunction $\phi : fg \Rightarrow \id_Y$ is part of a condensation $fg \condense \id_Y$ in $\End_\cC(Y)$.
\end{definition}
\begin{remark}
The data of $\gamma$ is not part of the data of a separable adjunction --- only its existence is required. A choice of $\gamma$ might be called a \define{separation} of the adjunction $f \dashv g$, in which case what we have called a ``unital condensation'' could equivalently be called a \define{separated}, as opposed to separable, adjunction.
\end{remark}

Separable adjunctions are fairly common in nature. However, if a $1$-category $\cC$ is treated as a $2$-category with only identity $2$-morphisms, then separable adjunctions in $\cC$ end up not corresponding to split surjections, but rather to isomorphisms. This is why we prefer condensations over  separable adjunctions as a categorification of ``split surjection.'' Nevertheless, we will show that the two notions often lead to equivalent ``Karoubi envelopes.''

\begin{proposition}\label{prop.unitalcondensation}
  Suppose that $(X,Y,f,g,\phi,\gamma,\dots)$ is a condensation and that $f$ admits a right adjoint $f^R$. Then there is a unital condensation $X \condense Y$ whose $1$-morphisms are $f$ and~$f^R$. If $g$ admits a left adjoint $g^L$, then there is a unital condensation $X \condense Y$ with $1$-morphisms $g^L$ and~$g$. Similarly, if $f$ admits a left adjoint $f^L$, then $f$ and~$f^L$ participate in a counital condensation $X \condense Y$, and if $g$ admits a right adjoint $g^R$, then $g$ and~$g^R$ participate in a counital condensation $X \condense Y$.
\end{proposition}
\begin{proof}
  We will prove the first statement; the others follow by reversing the order of composition of $1$- or $2$-morphisms. Our goal is to construct a condensation $ff^R \condense \id_Y$.
  
  Let $\eta : \id_X \Rightarrow f^R f$ and $\epsilon : f f^R \Rightarrow \id_Y$ denote the unit and counit of the adjunction $f \dashv f^R$. Let $\tilde\gamma$ denote the composition
  $$ \id_Y \overset{\gamma}\Longrightarrow fg \overset{f\eta g}\Longrightarrow f f^R f g \overset{ff^R\phi}\Longrightarrow ff^R. $$
  We claim that $\eta$ and $\tilde\gamma$ participate in a condensation $ff^R \condense \id_Y$. 
  Indeed, the following ``string diagram'' computation (with 1-morphisms drawn from right to left, and 2-morphisms from bottom to top) makes clear that the compositions $\epsilon \circ \tilde \gamma$ and $\phi \circ \gamma$ agree, and so to complete the condensation $ff^R \condense \id_Y$ we may use the condensation $\epsilon \circ \tilde \gamma = \phi \circ \gamma \condense \id_{\id_Y}$.
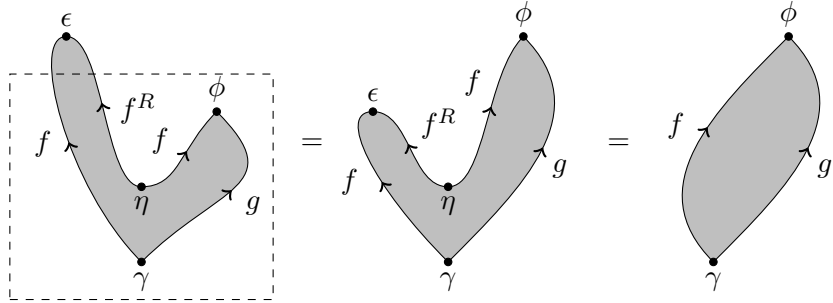
  $$ 
  \begin{tikzpicture}[baseline=(middle)]
    \path 
      (0,0) node[dot] (gamma) {} node[anchor=north] {$\gamma$}
      (0,1) node[dot] (eta) {} node[anchor=north] {$\eta$}
      (1,2) node[dot] (phi) {} node[anchor=south] {$\phi$}
      (-1,3) node[dot] (epsilon) {} node[anchor=south] {$\epsilon$}
      (0,1.5) coordinate (middle)
    ;
    \draw[onearrow] (gamma.center) .. controls +(1,1) and +(1,-1) .. node[auto,swap] {$g$}  (phi.center);
    \draw[onearrow] (eta.center) .. controls +(.5,0) and +(-.5,-.5) .. node[auto] {$f$} (phi.center);
    \draw[onearrow] (eta.center) .. controls +(-.5,0) and +(.5,0) .. node[auto,swap,pos=.4] {$f^R$} (epsilon.center);
    \draw[onearrow] (gamma.center) .. controls +(-1,1) and +(-.5,0) .. node[auto] {$f$} (epsilon.center);
    \draw[dashed] (-1.75,-.5) rectangle (1.75,2.5); 
    \begin{pgfonlayer}{background}
    \fill[gray!50] 
      (gamma.center) .. controls +(1,1) and +(1,-1) .. (phi.center)
      .. controls +(-.5,-.5) and +(.5,0) .. (eta.center)
      .. controls +(-.5,0) and +(.5,0) .. (epsilon.center)
      .. controls +(-.5,0) and +(-1,1) .. (gamma.center)
    ;
    \end{pgfonlayer}
  \end{tikzpicture}
  =
  \begin{tikzpicture}[baseline=(middle)]
    \path 
      (0,0) node[dot] (gamma) {} node[anchor=north] {$\gamma$}
      (0,1) node[dot] (eta) {} node[anchor=north] {$\eta$}
      (1,3) node[dot] (phi) {} node[anchor=south] {$\phi$}
      (-1,2) node[dot] (epsilon) {} node[anchor=south] {$\epsilon$}
      (0,1.5) coordinate (middle)
    ;
    \draw[onearrow] (gamma.center) .. controls +(1,1) and +(1,-1) .. node[auto,swap] {$g$}  (phi.center);
    \draw[onearrow] (eta.center) .. controls +(.5,0) and +(-.5,-.5) .. node[auto, pos=.6] {$f$} (phi.center);
    \draw[onearrow] (eta.center) .. controls +(-.5,0) and +(.5,0) .. node[auto,swap] {$f^R$} (epsilon.center);
    \draw[onearrow] (gamma.center) .. controls +(-1,1) and +(-.5,0) .. node[auto] {$f$} (epsilon.center);
    \begin{pgfonlayer}{background}
    \fill[gray!50] 
      (gamma.center) .. controls +(1,1) and +(1,-1) .. (phi.center)
      .. controls +(-.5,-.5) and +(.5,0) .. (eta.center)
      .. controls +(-.5,0) and +(.5,0) .. (epsilon.center)
      .. controls +(-.5,0) and +(-1,1) .. (gamma.center)
    ;
    \end{pgfonlayer}
  \end{tikzpicture}
  =
  \begin{tikzpicture}[baseline=(middle)]
    \path 
      (0,0) node[dot] (gamma) {} node[anchor=north] {$\gamma$}
      (1,3) node[dot] (phi) {} node[anchor=south] {$\phi$}
      (0,1.5) coordinate (middle)
    ;
    \draw[onearrow] (gamma.center) .. controls +(1,1) and +(1,-1) .. node[auto,swap] {$g$}  (phi.center);
    \draw[onearrow] (gamma.center) .. controls +(-1,1) and +(-1,-1) .. node[auto] {$f$} (phi.center);
    \begin{pgfonlayer}{background}
    \fill[gray!50] 
      (gamma.center) .. controls +(1,1) and +(1,-1) .. (phi.center)
      .. controls +(-1,-1) and +(-1,1) .. (gamma.center)
    ;
    \end{pgfonlayer}
  \end{tikzpicture}
  $$
  The white exterior region in each diagram is the object $X$, and the shaded interior region is the object $Y$.
  The boxed subdiagram on the left is the string diagram for $\tilde\gamma$. The second equality is simply the fact that $\eta$ and $\epsilon$ are the unit and counit of an adjunction between $f$ and $f^R$.
\end{proof}

\begin{remark}
 In the proof of Proposition~\ref{prop.unitalcondensation}, and especially in the proof of Theorem~\ref{thm.unitalization} below, we have doubled down on our convention to write ``$=$'' to mean ``isomorphic, canonically up to a contractible space of choices.'' This is consistent with the ``principle of univalence'' that should be built into any model-independent theory of higher categories. Readers uncomfortable with this principle can decide that $\cC$ is merely a 2-category: the results remain nontrivial, and compare our theory of 2-dimensional Karoubi completion to the version from \cite{Reutter2018}.\end{remark}

Given either a separable adjunction or a unital condensation (aka separated adjunction), one may consider the endomorphism $e = gf : X \to X$. 
We observed already that the morphism $\phi : fg \to \id_Y$ makes $e$ into an associative algebra object in $\End(X)$; the unit $\eta$ is precisely what is needed to make $e$ into a unital algebra. It is a general fact, true at arbitrary category number, that if an associative algebra admits a unit, then the unit is unique (up to a contractible space of choices): \emph{unitality} is a property. We are therefore justified to define a \define{unital condensation monad} to be a condensation monad whose associative multiplication is unital. Dually, a \define{counital condensation monad} is one whose coassociative comultiplication is counital. A \define{separable monad} is a unital associative monad which can be extended to a unital condensation monad.

\begin{remark}\label{remark.unitalmodule}
  Suppose that $e$ is a unital condensation monad and that $m$ is a (left, say) condensation $e$-module. The definition of condensation module does not mention the unit in $e$, and in the associative case it is easy to construct examples of a unital associative algebra with a module-as-a-nonunital-associative-algebra for which the unit does not act as the identity. Nevertheless, in the condensation case the unit in $e$ automatically acts as the identity on $m$.  Indeed, the unit in $e$ acts as the identity on the composition $e \otimes_e m$, because it acts on the composition through its action on $e$, but this composition is isomorphic to $m$.
\end{remark}

\begin{theorem}\label{thm.unitalization}
  Let $\cC$ be an $n$-category all of whose hom $(n-1)$-categories have all condensates.
  Suppose that $(X,e,\dots)$ is a condensation monad in $\cC$ whose underlying $1$-morphism $e : X\to X$ has a right adjoint $e^R$. Then $(X,e,\dots)$ is equivalent as an object of $\Kar(\cC)$ to an object represented by a unital condensation monad, and also to an object represented by a counital condensation monad.
\end{theorem}

Since objects and morphisms in $\Kar(\cC)$ are types of algebras and bimodules, equivalence in $\Kar(\cC)$ is a type of Morita equivalence.

\begin{proof}
  Through only the most mild abuse of notation, let us write $X \in \Kar(\cC)$ for the object represented by the condensation monad $(X,\id,\dots)$, and $e\in \Kar(\cC)$ for the object represented by $(X,e,\dots)$. Theorem~\ref{thm.Kar} provides a condensation $X \condense e$ in $\Kar(\cC)$. The $1$-morphisms $f: X \to e$ and $g : e \to X$ are, respectively, $e$ thought of as a left condensation $e$-module and $e$ thought of as a right condensation $e$-module.
  
  We will show that the $1$-morphism $f: X \to e$ in $\Kar(\cC)$ admits a right adjoint $f^R$. A similar construction produces a right adjoint to $g$. Before giving the construction, we explain why these right adjoints suffice to prove the Theorem. Indeed, by Proposition~\ref{prop.unitalcondensation}, the right adjoint $f^R$ provides a unital condensation $X \condense e$ in $\Kar(\cC)$, which is to say a unital condensation monad $e' = f^R \otimes_e f : X \to X$ in $\Kar(\cC)$. The inclusion $\cC \to \Kar(\cC)$ is fully faithful, and so $e'$ determines a unital condensation monad in the original category $\cC$. But $e'$ is constructed from a condensation $X \to e$, and so it condenses (in $\Kar(\cC)$) to $e$, and on the other hand it automatically condenses to $e'$. Theorem~\ref{thm.uniqueness} finishes the job: $e$ is equivalent in $\Kar(\cC)$ to the unital condensation monad $e'$. By using $g^R$ instead, we would have constructed a counital $e'$.
  
  It remains to construct $f^R$. It is supposed to be a morphism $e\to X$ in $\Kar(\cC)$, which is to say it is supposed to be a right condensation $e$-module. Since $f$ is ``$e$ as an $e$-module,'' and since $e$ admits a right-adjoint $e^R$, it is tempting to think that $f^R$ will be ``$e^R$ as an $e$-module.'' In fact, $e^R$ is naturally a right comodule for the (noncounital) coalgebra structure on $e$, with coaction related by the adjunction to the action of $e$ on itself, but it is not naturally a right module for the algebra structure, and hence not naturally a condensation module. (On the left, it is naturally a module but not a comodule.) 
  
  That said, the natural comodule structure on $e^R$ is enough to define a version of the tensor product ``$e^R \otimes_e e$.'' This tensor product will be our condensation module $f^R$. Specifically, $e^Re$ supports a condensation monad whose underlying $2$-morphism $e^Re\Rightarrow e^Re$ is given by the following string diagram. (As above, our conventions for string diagrams is that composition of $1$-morphisms is from right to left, and composition of $2$-morphisms is from bottom to top.)
  
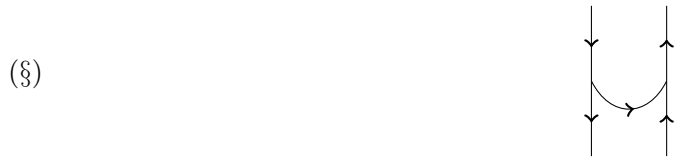
\begin{equation} \label{eqn.unitalization}
    \begin{tikzpicture}[baseline=(basepoint)]
      \draw[onearrow] (0,2) -- (0,1); \draw[onearrow] (0,1) -- (0,0); 
      \draw[onearrow] (1,0) -- (1,1); \draw[onearrow] (1,1) -- (1,2); 
      \draw[onearrow] (0,1) .. controls +(.25,-.5) and +(-.25,-.5) .. (1,1);
      \path (0,1) coordinate (basepoint);
    \end{tikzpicture}
  \end{equation}
  The upward-pointing edges denote $e$, and the downward-pointing edges are $e^R$. The cup is the unit of the adjunction between $e$ and $e^R$.
  The trivalent vertices are the multiplication and comultiplication on $e$, rotated using the adjunction $e \dashv e^R$.
  The tensor product $f^R = \text{``}e^R \otimes_e e\text{''}$ is the condensate produced from this condensation monad. 
  We emphasize that $f^R$ and $e^R$ are typically not the same as $1$-morphisms: the isomorphism $m \otimes_e e \cong m$ holds when $m$ is a condensation $e$-module, or when $e$ is a unital algebra (with unital action on $m$), but not for arbitrary modules of a nonunital algebra.
  
  It is clear that $f^R$ is a right condensation $e$-module: the right condensation action of $e$ on itself commutes with the left action used to define $f^R$. 
  We will henceforth use white regions to denote the object $X \in \Kar(\cC)$, and shading to denote $e$, so that an edge with shading on the side means a (bi)module. So for example the composition $f^R \otimes_e f$ will be denoted:
  $$
    
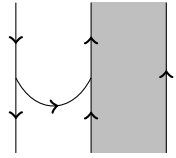
\begin{tikzpicture}
      \fill[gray!50] (1,0) rectangle (2,2);
      \draw[onearrow] (0,2) -- (0,1); \draw[onearrow] (0,1) -- (0,0); 
      \draw[onearrow] (1,0) -- (1,1); \draw[onearrow] (1,1) -- (1,2); 
      \draw[onearrow] (0,1) .. controls +(.25,-.5) and +(-.25,-.5) .. (1,1);
      \draw[onearrow] (2,0) -- (2,2); %\draw[onearrow] (2,1.5) -- (2,2);
    \end{tikzpicture}
  $$
 Recall from Definition~\ref{defn.tensorproduct} that $\otimes_e$ is implementing by condensing a network of $e$-edges, so that for instance
 \,$\begin{tikzpicture}[baseline=(basepoint)]
    \path(0,2pt) coordinate (basepoint);
    \fill[gray!50] (0, -1pt) rectangle (2ex,12pt);
    \draw[onearrow] (0,-1pt) -- (0,12pt);
    \draw[onearrow] (2ex,-1pt) -- (2ex,12pt);
  \end{tikzpicture}
  =
  \begin{tikzpicture}[baseline=(basepoint)]
    \path(0,2pt) coordinate (basepoint);
    \draw[onearrow] (0,-1pt) -- (0,12pt);
    \draw[onearrow] (2ex,-1pt) -- (2ex,12pt);
    \draw (0,8pt) -- (2ex,8pt);
  \end{tikzpicture}
  =
  \begin{tikzpicture}[baseline=(basepoint)]
    \path(0,2pt) coordinate (basepoint);
    \draw[onearrow] (0,-1pt) -- (0,12pt);
  \end{tikzpicture}
  $\;.
  In particular, $f^R \otimes_e f$ is simply the underlying $1$-morphism $X \to X$ of $f^R$, without its structure as a right condensation $e$-module.

  To complete the proof, we must witness the adjunction $f \dashv f^R$, i.e.\ we must give unit and counit maps $\id_X \Rightarrow f^R \otimes_e f$ and $f \otimes_X f^R \Rightarrow \id_e$.   
  The unit $\id_X \Rightarrow f^R \otimes_e f$ of the adjunction is  given by splitting a ``bubble collapsing'' condensation monad supported by the following diagram:
  $$
    \begin{tikzpicture}
      \fill[gray!50] (1.5,-1) .. controls +(-.25,.5) and +(0,-.5) .. (1,0) -- (1,2) -- (2,2) -- (2,0) .. controls +(0,-.5) and +(.25,.5)  .. (1.5,-1);
      \draw[onearrow] (0,2) -- (0,1); \draw[onearrow] (0,1) -- (0,0); 
      \draw[onearrow] (1,0) -- (1,1); \draw[onearrow] (1,1) -- (1,2); 
      \draw[onearrow] (0,1) .. controls +(.25,-.5) and +(-.25,-.5) .. (1,1);
      \draw[onearrow] (2,0) -- (2,2);
      \draw[onearrow] (1.5,-1) .. controls +(-.25,.5) and +(0,-.5) .. (1,0);
      \draw[onearrow] (1.5,-1) .. controls +(.25,.5) and +(0,-.5) .. (2,0);
      \draw[onearrow] (0,0) .. controls +(0,-1) and +(0,-.5) .. (1.5,-1);
    \end{tikzpicture}
  $$
  The counit $f \otimes_X f^R \Rightarrow \id_e$ is the splitting of a ``bubble collapsing'' condensation monad supported by:
  $$
    \begin{tikzpicture}
      \fill[gray!50] (1,0) -- (1,1) -- (1,2) .. controls +(0,1) and +(.25,-.5) .. (0,3) -- (0,4) -- (2,4) -- (2,0) -- (1,0);
      \fill[gray!50] (-1,0) -- (-1,1) -- (-1,2) .. controls +(0,1) and +(-.25,-.5) .. (0,3) -- (0,4) -- (-2,4) -- (-2,0) -- (-1,0);
      \draw[onearrow] (0,2) -- (0,1); \draw[onearrow] (0,1) -- (0,0); 
      \draw[onearrow] (1,0) -- (1,1); \draw[onearrow] (1,1) -- (1,2); 
      \draw[onearrow] (0,1) .. controls +(.25,-.5) and +(-.25,-.5) .. (1,1);
      \draw[onearrow] (-1,0) -- (-1,2); 
      \draw[onearrow] (-1,2) .. controls +(.25,.5) and +(0,.5) .. (0,2);
      \draw[onearrow] (-1,2) .. controls +(0,1) and +(-.25,-.5) .. (0,3);
      \draw[onearrow] (1,2) .. controls +(0,1) and +(.25,-.5) .. (0,3);
      \draw[onearrow] (0,3) -- (0,4);
    \end{tikzpicture}
  $$
  Note that 
  \begin{tikzpicture}[baseline=(basepoint)]
    \path(0,2pt) coordinate (basepoint);
    \fill[gray!50] (-2ex, -1pt) rectangle (2ex,12pt);
    \draw[onearrow] (0,-1pt) -- (0,12pt);
  \end{tikzpicture}
  is a picture of the ``$e$ as an $e$-bimodule,'' which is to say the identity 1-morphism $\id_e$ in $\Kar(\cC)$.
  
  A diagrammatic calculation (left to the reader) confirms that these are in fact the unit and counit of an adjunction $f \dashv f^R$.
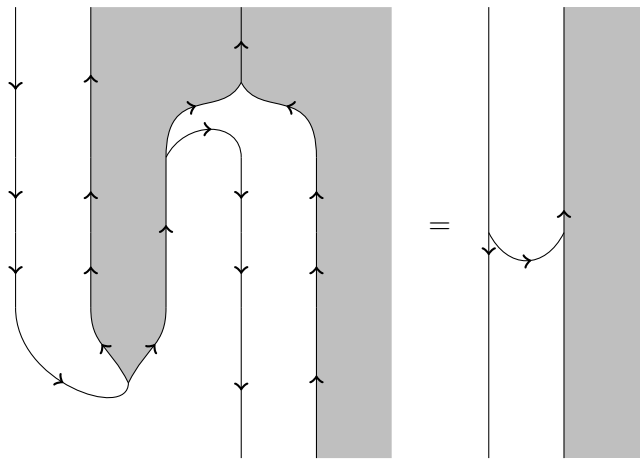
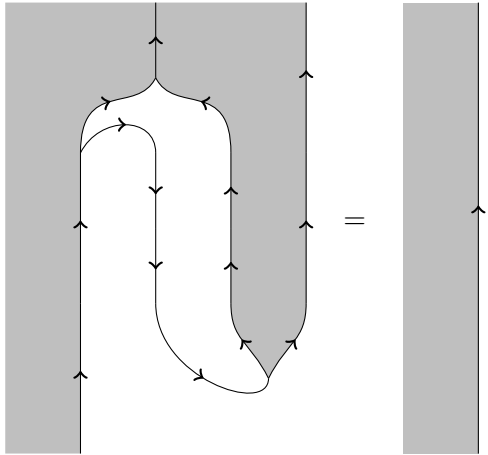
  $$ 
    \begin{tikzpicture}[baseline=(basepoint)]
      \path (0,1) coordinate (basepoint);
      \fill[gray!50] (4,-2) -- (4,0) -- (4,1) -- (4,2) .. controls +(0,1) and +(.25,-.5) .. (3,3) -- (3,4) -- (5,4) -- (5,-2) -- (4,-2);
      \fill[gray!50] (2,0) -- (2,1) -- (2,2) .. controls +(0,1) and +(-.25,-.5) .. (3,3) -- (3,4) -- (1,4) -- (1,0) -- (2,0);
      \fill[gray!50] (1.5,-1) .. controls +(-.25,.5) and +(0,-.5) .. (1,0) -- (1,2) -- (2,2) -- (2,0) .. controls +(0,-.5) and +(.25,.5)  .. (1.5,-1);
      \draw[onearrow] (0,2) -- (0,1); \draw[onearrow] (0,1) -- (0,0); 
      \draw[onearrow] (1,0) -- (1,1); \draw[onearrow] (1,1) -- (1,2); 
      \draw[onearrow] (2,0) -- (2,2);
      \draw[onearrow] (1.5,-1) .. controls +(-.25,.5) and +(0,-.5) .. (1,0);
      \draw[onearrow] (1.5,-1) .. controls +(.25,.5) and +(0,-.5) .. (2,0);
      \draw[onearrow] (0,0) .. controls +(0,-1) and +(0,-.5) .. (1.5,-1);
      \draw[onearrow] (3,2) -- (3,1); \draw[onearrow] (3,1) -- (3,0); 
      \draw[onearrow] (4,0) -- (4,1); \draw[onearrow] (4,1) -- (4,2); 
      \draw[onearrow] (2,2) .. controls +(.25,.5) and +(0,.5) .. (3,2);
      \draw[onearrow] (2,2) .. controls +(0,1) and +(-.25,-.5) .. (3,3);
      \draw[onearrow] (4,2) .. controls +(0,1) and +(.25,-.5) .. (3,3);
      \draw[onearrow] (3,3) -- (3,4);
      \draw[onearrow] (0,4) -- (0,2); \draw[onearrow] (1,2) -- (1,4); \draw[onearrow] (3,0) -- (3,-2); \draw[onearrow] (4,-2) -- (4,0);
      \draw[onearrow] (0,1) .. controls +(.25,-.5) and +(-.25,-.5) .. (1,1);
      \draw[onearrow] (3,1) .. controls +(.25,-.5) and +(-.25,-.5) .. (4,1);
    \end{tikzpicture}
    \quad \condense \quad 
    \begin{tikzpicture}[baseline=(basepoint)]
      \path (0,1) coordinate (basepoint);
      \fill[gray!50] (1,-2) rectangle (2,4);
      \draw[onearrow] (0,4) -- (0,-2); \draw[onearrow] (1,-2) -- (1,4);
      \draw[onearrow] (0,1) .. controls +(.25,-.5) and +(-.25,-.5) .. (1,1);
    \end{tikzpicture}
    $$ $$
    \begin{tikzpicture}[baseline=(basepoint)]
      \path (0,1) coordinate (basepoint);
      \fill[gray!50] (1,0) -- (1,1) -- (1,2) .. controls +(0,1) and +(.25,-.5) .. (0,3) -- (0,4) -- (2,4) -- (2,0) -- (1,0);
      \fill[gray!50] (-1,-2) -- (-1,1) -- (-1,2) .. controls +(0,1) and +(-.25,-.5) .. (0,3) -- (0,4) -- (-2,4) -- (-2,-2) -- (-1,-2);
      \fill[gray!50] (1.5,-1) .. controls +(-.25,.5) and +(0,-.5) .. (1,0) -- (1,2) -- (2,2) -- (2,0) .. controls +(0,-.5) and +(.25,.5)  .. (1.5,-1);
      \draw[onearrow] (0,2) -- (0,1); \draw[onearrow] (0,1) -- (0,0); 
      \draw[onearrow] (1,0) -- (1,1); \draw[onearrow] (1,1) -- (1,2); 
      \draw[onearrow] (0,1) .. controls +(.25,-.5) and +(-.25,-.5) .. (1,1);
      \draw[onearrow] (-1,0) -- (-1,2); 
      \draw[onearrow] (-1,2) .. controls +(.25,.5) and +(0,.5) .. (0,2);
      \draw[onearrow] (-1,2) .. controls +(0,1) and +(-.25,-.5) .. (0,3);
      \draw[onearrow] (1,2) .. controls +(0,1) and +(.25,-.5) .. (0,3);
      \draw[onearrow] (0,3) -- (0,4);
      \draw[onearrow] (2,0) -- (2,2);
      \draw[onearrow] (1.5,-1) .. controls +(-.25,.5) and +(0,-.5) .. (1,0);
      \draw[onearrow] (1.5,-1) .. controls +(.25,.5) and +(0,-.5) .. (2,0);
      \draw[onearrow] (0,0) .. controls +(0,-1) and +(0,-.5) .. (1.5,-1);
      \draw[onearrow] (2,2) -- (2,4); \draw[onearrow] (-1,-2) -- (-1,0);
    \end{tikzpicture}
    \quad \condense \quad 
    \begin{tikzpicture}[baseline=(basepoint)]
      \path (0,1) coordinate (basepoint);
      \fill[gray!50] (0,-2) rectangle (1,4);
      \draw[onearrow] (1,-2) -- (1,4);
    \end{tikzpicture}
  $$
  In both cases the calculation uses the adjunction $e \dashv e^R$, the Frobenius axtioms for $e$, and the ``bubble collapsing''  condensation monads (not drawn).
\end{proof}

\subsection{Physical interpretation: the state-operator map} \label{subsec.stateoperatorcorresp}

Suppose $X$ and $Y$ are two gapped topological $(1+1)$-dimensional phases and that we have a way of (physically) condensing from $X$ to $Y$.
In \S\ref{sec.bc}
we constructed a (categorical) condensation $X \condense Y$ by partially condensing in the $x>0$ or $x<0$ domains. The interfaces $f$ and $g$ produced by this procedure cannot be literally ``the same'' because they point in different directions. In a TQFT context, it would be natural to request that they are related by $180^\circ$ rotation. This request is much less natural in the condensed matter context, where any notion of ``rotation of an interface'' is at best emergent in the low-energy. 
We have mentioned already that the ``topological'' limits of gapped topological condensed matter phases are, a priori, actually framed-topological. It is certainly feasible that there may be condensation procedures that couple nontrivially to the microscopic lattice framing. In this case it can happen that counterclockwise $180^\circ$-rotation by is not equivalent to clockwise $180^\circ$-rotation, even if both rotations make sense. The former produces the left adjoint $f^L$ of an interface $f$, whereas the latter produces the right adjoint $f^R$.
That we may have $f^L \neq f^R$ illustrates the a priori unnaturalness of demanding that $g$ equal either of them.
Nevertheless, Proposition~\ref{prop.unitalcondensation} shows that, assuming these interfaces can be rotated, then we lose no generality by requesting that $g = f^R$, or that $g = f^L$, if we so desire.

Similarly, unitality for a condensation monad is natural in the (oriented) TQFT context, where the space of states on a segment (provided that ``the same'' boundary condition is used at both ends of the segment) is naturally equipped not just with a multiplication map (given by adiabatically merging two segments to one) and a comultiplication map (adiabatically cutting a segment in two) satisfying the Frobenius axioms, but also with a unit and counit corresponding to half-disk geometries.
Of course, these half disks only exist if we use ``the same'' boundary condition for the two ends of the segment (i.e.\ if $g = f^L = f^R$).

Suppose we are given a nonunital condensation algebra $e$, and we construct, following Theorem~\ref{thm.unitalization}, a Morita-equivalent unital condensation algebra $e' = e^R \otimes_e e$. By opening up the top of~(\ref{eqn.unitalization}), we can recognize $e'$ as the space of insertions at ``$\ast$'' in the following diagram:
$$
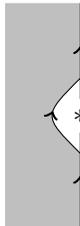
\begin{tikzpicture}
  \fill[gray!50] (0,-1.5) -- (0,-.5) .. controls +(-.5,.5) and +(-.5,-.5) .. (0,.5) -- (0,1.5) -- (-1,1.5) -- (-1,-1.5) -- (0,-1.5);
  \path (0,0) node (X) {$\ast$};
  \draw[onearrow] (0,-1.5) -- (X);
  \draw[onearrow] (X) -- (0,1.5);
  \draw[onearrow] (0,-.5) .. controls +(-.5,.5) and +(-.5,-.5) .. (0,.5);
\end{tikzpicture}
$$
As above, the shaded region denotes the phase produced by condensing the $e$-anyon. For comparison, the original nonunital monad $e$ is the space of states for the condensed system.

But this $e'$ is precisely the space of \define{boundary operators} for the boundary condition \,$\begin{tikzpicture}[baseline=(basepoint)]
    \path(0,2pt) coordinate (basepoint);
    \fill[gray!50] (0, -1pt) rectangle (2ex,12pt);
    \draw[onearrow] (2ex,-1pt) -- (2ex,12pt);
  \end{tikzpicture}$\,. The multiplication on $e'$ is precisely the multiplication of boundary operators.
Furthermore, there is a \define{state-operator map} $e \to e'$ described by the string diagram
$$    
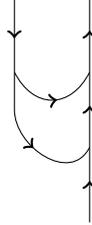
\begin{tikzpicture}[baseline=(basepoint)]
      \draw[onearrow] (0,1) .. controls +(.25,-.5) and +(-.25,-.5) .. (1,1);
      \draw[onearrow] (0,2) -- (0,1);   \draw[onearrow] (1,1) -- (1,2);
      \draw[onearrow] (1,0) -- (1,1);   \draw[onearrow] (1,-1) -- (1,0);
      \draw[onearrow] (0,1) -- (0,.5) .. controls +(0,-.5) and +(-.25,-.5) .. (1,0);
    \end{tikzpicture}
$$
which is an isomorphism if and only if $e$ was already unital --- only in the unital case do we get a true ``state-operator correspondence.''

\begin{remark}
  Any nonunital associative algebra $A$, not necessarily a condensation algebra, determines a unital associative algebra, called the \define{multiplier algebra of $A$}, defined as the algebra $\End_A(A)$ of endomorphisms of $A$ thought of as a left $A$-module.  The multiplication map determines an associative algebra morphism $A \to \End_A(A)$. The construction $e \leadsto e'$ from Theorem~\ref{thm.unitalization} is precisely this construction $A \leadsto \End_A(A)$ in the special case where $e$ is a (right-adjunctible) condensation algebra.
  
 The Morita equivalence $e \simeq e'$ from Theorem~\ref{thm.unitalization} is special to the condensation case: for general associative algebras there is no sense in which $A$ and $\End_A(A)$ are Morita equivalent.
\end{remark}

\subsection{Condensation bimodules vs unital algebra bimodules}

Theorem~\ref{thm.unitalization} allows us in many cases to replace condensation algebras with unital condensation algebras. In order to get a comparison with unital separable algebras, 
we must show that the choice of comultiplication can be dropped.

\begin{lemma}\label{lemma.extendingmodulestructure}
  Let $e$ be a condensation monad in an $n$-category $\cC$. Suppose that $m$ is left $e$-module with respect to the (nonunital) associative algebra structure on $e$. If the $e$-module structure on $m$ extends to a condensation $e$-module structure, then it does so in a unique way (up to a contractible space of choices).
\end{lemma}

\begin{proof}
  By mimicking the proof of Theorem~\ref{thm.unitalization}, we can construct a ``tensor product'' $\tilde m = \text{``}e \otimes_e m\text{''}$ which will be automatically a left condensation $e$-module. Indeed, consider the ($2$-)morphism $e \circ m \Rightarrow e \circ m$ defined by the following string diagram:
  $$   
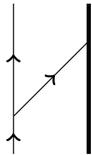
\begin{tikzpicture}[baseline=(basepoint)]
   \path (0,.9) coordinate (basepoint);
   \draw[onearrow] (0,0) -- (0,.5);
   \draw[onearrow] (0,.5) -- (0,2);
   \draw[onearrow] (0,.5) -- (1,1.5);
   \draw[ultra thick] (1,0) -- (1,2);
  \end{tikzpicture}
 $$
 Here and throughout, $\; 
  \begin{tikzpicture}[baseline=(basepoint)]
    \path(0,2pt) coordinate (basepoint);
    \draw[onearrow] (0,-1pt) -- (0,12pt);
  \end{tikzpicture}\;$ denotes the condensation monad $e$ and $\; 
  \begin{tikzpicture}[baseline=(basepoint)]
    \path(0,2pt) coordinate (basepoint);
    \draw[ultra thick] (0,-1pt) -- (0,12pt);
  \end{tikzpicture}\;$ denotes $m$, and the trivalent vertices denote various (co)multiplications and actions. Just using coassociativity of  $
  \; 
  \begin{tikzpicture}[baseline=(basepoint)]
    \path(0,2pt) coordinate (basepoint);
    \draw[onearrow] (0,-1pt) -- (0,5pt); \draw[onearrow] (0,5pt) -- (-4pt,12pt); \draw[onearrow] (0,5pt) -- (4pt,12pt);
  \end{tikzpicture}\;$, associativity of the action $ \; 
  \begin{tikzpicture}[baseline=(basepoint)]
    \path(0,2pt) coordinate (basepoint);
    \draw[onearrow] (-4pt,-1pt) -- (4pt,7pt); \draw[ultra thick] (4pt,-1pt) -- (4pt,12pt);
  \end{tikzpicture}\;$, and the specialness axiom for $e$ (see Remark~\ref{remark.special} for the $n>2$ case), one can produce on this (2-)morphism a condensation monad structure. Condensing this monad defines $\tilde m = e \otimes_e m$.
  
  When the $e$-action on $m$ extends to a condensation action already, then $\tilde m$ will be isomorphic, as a condensation module, to $m$. Thus all extensions, if they exist, are canonically equivalent.
\end{proof}

\begin{lemma}\label{lemma.extendingmonadstructure}
  Let $e$ be a (nonunital) associative monad in $\cC$. If the associative monad structure on $e$ extends to a condensation monad structure, then all extensions are canonically ``condensation Morita equivalent,'' i.e.\ canonically equivalent as objects of $\Kar(\cC)$.
\end{lemma}

\begin{proof}
  Suppose that $e_1$ and $e_2$ are two condensation monads whose underlying associative monads are both  $e$. Consider $e$ as an associative $e$-$e$-bimodule. 
  Then, as in the proof of Lemma~\ref{lemma.extendingmodulestructure}, construct the tensor product $e_1 \otimes_{e_1} e \otimes_{e_2} e_2$. It is by construction a condensation $e_1$-$e_2$-bimodule. But its underlying associative module is simply $e$ as an $e$-$e$-module, just as in Lemma~\ref{lemma.extendingmodulestructure}, and so we have constructed a (unique) condensation bimodule extension. Similarly, the tensor product $e_2 \otimes_{e_2} e \otimes_{e_1} e_1$ determines a condensation $e_2$-$e_1$-bimodule structure on $e$. Finally, the tensor products $(e_2 \otimes_{e_2} e \otimes_{e_1} e_1) \otimes_{e_1} (e_1 \otimes_{e_1} e \otimes_{e_2} e_2)$ and $(e_1 \otimes_{e_1} e \otimes_{e_2} e_2) \otimes_{e_2} (e_2 \otimes_{e_2} e \otimes_{e_1} e_1)$ are easily seen to be identities (since they are condensation bimodule extensions of $e$ to an $e_1$-$e_1$- or $e_2$-$e_2$-bimodule, and such extension is unique).
\end{proof}

A monoidal $1$-category is called \define{rigid} if all its objects admit both left and right duals. 
We will call a monoidal $n$-category is \define{$1$-rigid} if all its objects admit left and right duals. (Compare \S\ref{subsec.fulldual}.)
Combining Theorem~\ref{thm.unitalization} with Lemmas~\ref{lemma.extendingmodulestructure} and~\ref{lemma.extendingmonadstructure} gives:

\begin{theorem}\label{thm.reutter}
  Let $\cC$ be a $n$-category such that all hom $(n-1)$-categories in $\cC$ have all condensates and all
   endomorphism $(n-1)$-categories in $\cC$ are $1$-rigid. Then $\Kar(\cC)$ is equivalent to the $n$-category $\cC^\triangledown$ whose objects are separable monads in $\cC$ and whose $n$-morphisms are bimodules in the associative unital sense.
\end{theorem}
Note that separable monads correspond to separable adjunctions, and so are automatically unital. The name $\cC^\triangledown$ for the latter category comes from \cite{Reutter2018}.

\begin{proof}
  Let $\Kar^{\mathrm{un}}(\cC)$ denote the full sub-$n$-category of $\Kar(\cC)$ on the objects represented by unital condensation monads. I.e.\ an object of $\Kar^{\mathrm{un}}(\cC)$ is a unital condensation monad, and a $1$-morphism is a condensation bimodule. 
  Under the condition that all endomorphism categories in $\cC$ are rigid, Theorem~\ref{thm.unitalization} implies that
  the inclusion $\Kar^{\mathrm{un}}(\cC) \to \Kar(\cC)$ is essentially surjective, and so an equivalence. 
  
  We will show that the forgetful functor $\Kar^{\mathrm{un}}(\cC) \to \cC^\triangledown$ which forgets the comultiplication is also an equivalence.
  (That this functor is defined follows from Remark~\ref{remark.unitalmodule}, which showed that condensation bimodules between unital condensation monads automatically respect the units.)
   It is essentially surjective by definition of $\cC^\triangledown$, and so it suffices to show that $\Kar^{\mathrm{un}}(\cC) \to \cC^\triangledown$ induces equivalences on hom categories, i.e.\ that, given unital condensation monads $e_1$ and $e_2$, the forgetful functor $\{$condensation $e_1$-$e_2$-bimodules$\} \to \{$unital associative $e_1$-$e_2$-bimodules$\}$ is an equivalence. But suppose $m$ is a unital associative $e_1$-$e_2$-bimodule. Then, as in the proof of Lemma~\ref{lemma.extendingmodulestructure}, construct the condensation bimodule $e_1 \otimes_{e_1} m \otimes_{e_2} e_2$. Because $e_1$ and $e_2$ are unital, this tensor product is isomorphic (as a unital associative bimodule) to $m$. Thus the functor $m \mapsto e_1 \otimes_{e_1} m \otimes_{e_2} e_2$ is an inverse to the forgetful functor $\{$condensation $e_1$-$e_2$-bimodules$\} \to \{$unital associative $e_1$-$e_2$-bimodules$\}$, completing the proof.
\end{proof}

Recall from Example~\ref{example.sigmadefn} that, for a symmetric monoidal $n$-category $\cC$, there is a natural symmetric monoidal $(n+1)$-category $\Sigma\cC = \Kar(\rB \cC)$ and an equivalence $\cC \cong \Omega\Sigma\cC := \End_{\Sigma\cC}(1)$.

\begin{corollary} \label{cor.sigmavec}
  Let $\Vect$ denote the category of finite-dimensional vector spaces over a field~$\bK$. Then $\Sigma\Vect$ is equivalent to the $2$-category whose objects are finite-dimensional separable unital associative $\bK$-algebras, whose $1$-morphisms are finite-dimensional (unital associative) bimodules, and whose $2$-morphisms are homomorphisms of bimodules. \qed
\end{corollary}

We expect essentially the same result one dimension up, but we did not check the details:

\begin{conjecture}\label{conj.sigma2vec}
  $\Sigma^2\Vect$ is equivalent to the $3$-category, studied in \cite{DSPS,DSPS2}, whose objects are separable finite tensor categories and whose $1$-morphisms are their separable bimodule categories.
\end{conjecture}

Over a field of characteristic zero, the separable finite tensor categories are precisely the multifusion (i.e.\ semisimple tensor) categories, and a bimodule category is separable if and only if it is semisimple.

\section{Absolute limits and full dualizability}\label{sec.ab}

The goal of this final section is to generalize to $n$-categories the following two well-known results about Karoubi envelopes of $1$-categories:
\begin{enumerate}
  \item If $\cC$ is a $1$-category, then $\Kar(\cC)$ is the closure of $\cC$ under absolute limits, also called the  \define{Cauchy completion} of $\cC$. (A limit is \define{absolute} if limits of that shape are preserved by arbitrary functors.) \label{item.motivation1}
  \item Given a commutative ring $R$, denote by $\Mat(R)$ the category of finitely generated free $R$-modules. Then $\Kar(\Mat(R))$ is naturally equivalent to the category of dualizable $R$-modules. \label{item.motivation2}
\end{enumerate}
By generalizing (\ref{item.motivation1}) to $n$-categories, we fully justify condensations as the ``correct'' categorification of split surjections. By generalizing (\ref{item.motivation2}) we complete our comparison between gapped phases of matter and TQFTs.

\subsection{Full dualizability}\label{subsec.fulldual}

The following notions are by now quite standard in higher category theory. Let $\cC$ be an $n$-category and $f : X \to Y$ a $k$-morphism. 
A \define{right adjoint} of $f$ is a morphism $f^R : Y \to X$ together with \define{unit} and \define{counit} $(k+1)$-morphisms $\eta:  \id_X \Rightarrow f^R f$ and $\epsilon : ff^R \Rightarrow \id_Y$ such that the ``zig-zag'' compositions
\begin{gather*}
 f \overset{f\eta}\Longrightarrow ff^R f \overset{\epsilon f}\Longrightarrow f,\\
 f^R \overset{\eta f^R}\Longrightarrow f^R f f^R \overset{f^R \epsilon}\Longrightarrow  f^R,
\end{gather*}
are identities.
Similarly, a \define{left adjoint} of $f$ is a morphism $f^L : Y \to X$ together with $(k+1)$-morphisms $\eta : \id_Y \Rightarrow ff^L$ and $\epsilon : f^L f \Rightarrow \id_X$ such that the ``zig-zag'' compositions
\begin{gather*}
 f \overset{\eta f}\Longrightarrow ff^L f \overset{f\epsilon}\Longrightarrow f,\\
 f^L \overset{f^L\eta}\Longrightarrow f^L f f^L \overset{\epsilon f^L}\Longrightarrow f^L,
\end{gather*}
are identities. 

\begin{remark}
  There are various, ultimately equivalent, ways to implement the above definition, and in particular to interpret the phrase  ``such that (\dots) are identities,'' in higher categories. The version of adjunction given in \cite[Definition 2.3.13]{Lur09} simply asks for the existence of equivalences between the given compositions and the appropriate identities --- in other words, in this version the adjoints $f^R$ and $f^L$ of $f$ are defined in terms of the homotopy bicategory of $\cC$. (In fact, to show adjunctibility, it suffices to work in the gauntification of the homotopy bicategory of $\cC$ \cite[Lemma~7.9]{JFS}.)
  
  A fully coherent theory of adjunction is given in \cite{MR3415698}. The desideratum for a fully coherent theory is that the following well-known fact from 2-categories should remain true: if a left or right adjoint exists, then it should be unique up to a contractible space of choices. One way to achieve this is to ask, when witnessing a right adjoint $f^R$ to $f$, for the data of an equivalence $(\epsilon f)\circ(f\eta) \overset\sim \Rightarrow \id_f$ while asking $(f^R\epsilon)\circ(\eta f^R)$ to have the property of being equivalent to an identity. Or one can ask to supply both equivalences, but then one must also supply further coherence data relating them.
\end{remark}

It can happen that both $f^L$ and $f^R$ exist but that they are not isomorphic; even if they are isomorphic, it can happen that there is not a canonical isomorphism.
Clearly $(f^L)^R = (f^R)^L = f$. The notation ``$f \dashv f^R$'' means ``$f^R$ is a right adjoint to $f$.''
If $f,g : X \to Y$ both admit right adjoints, then the \define{mate} of a $(k+1)$-morphism $\phi : f \Rightarrow g$ is the composition
$$ g^R \overset{\eta_f g^R} \Longrightarrow f^R f g^R \overset{f^R \phi g^R}\Longrightarrow f^R g g^R \overset{f^R \epsilon_g}\Longrightarrow f^R,$$
where $\eta_f$ is the unit of the adjunction $f \dashv f^R$ and $\epsilon_g$ is the counit of the adjunction $g \dashv g^R$.

Suppose that $\cC$ is a monoidal $n$-category. The left and right \define{duals} to an object $X\in\cC$, if they exist, are the left and right adjoints to $X$ thought of as a $1$-morphism in the one-object $(n+1)$-category $\rB \cC$. A monoidal $1$-category is traditionally called \define{rigid} if all objects admit both left and right duals. Extending this notation, we will say that an $(n+1)$-category is \define{$m$-rigid} if all $k$-morphisms with $1\leq k \leq m$ admit both left and right adjoints. A monoidal $n$-category $\cC$ is \define{$m$-rigid} if the $n$-category $\rB \cC$ is $m$-rigid, i.e.\ if all $k$-morphisms in $\cC$ with $0\leq k < m$ admit (duals or) adjoints. An $(n+1)$-category or monoidal $n$-category is \define{fully rigid} if it is $n$-rigid, this being the maximal amount of rigidity a nontrivial $(n+1)$-category may enjoy --- an $(n+1)$-category is $(n+1)$-rigid if and only if all morphisms are invertible.

\begin{theorem}\label{thm.fullrigidity}
  Let $\cC$ be an $(n+1)$-category all of whose hom-$n$-categories have all condensates. If $\cC$ is $m$-rigid, then so is $\Kar(\cC)$. In particular, if $\cC$ is a fully rigid symmetric monoidal $n$-category with all condensates, then the symmetric monoidal $(n+1)$-category $\Sigma \cC$ is also fully rigid.
\end{theorem}

\begin{proof}
  Corollary~\ref{cor.dualizable}, which we will prove independently, implies that for any symmetric monoidal $n$-category $\cC$ with all condensates, every object in $\Sigma\cC$ is $1$-dualizable. (The dual to a condensation algebra in $\cC$ is its ``opposite'' condensation algebra.) The second sentence of the Theorem follows from this together with the first sentence: if $\cC$ is a fully rigid symmetric monoidal $n$-category, then $\rB \cC$ is a fully rigid symmetric monoidal $(n+1)$-category (and so all $k$-morphisms with $k\geq 1$ admit adjoints).
  
  We now turn to the first sentence of the Theorem. The main thing to show is that if $\cC$ is $1$-rigid, then so is $\Kar(\cC)$. Suppose  that $Y_1 \overset\mu\from Y_2$ is a $1$-morphism in $\Kar(\cC)$; we wish to construct its right adjoint $\mu^R$ (an analogous construction provides its left adjoint). By Theorem~\ref{thm.unitalization}, we may represent $Y_1$ by a unital condensation monad $(X_1,e_1,\dots)$ in $\cC$, and we may represent $Y_2$ by a counital condensation monad $(X_2,e_2,\dots)$ in $\cC$; with these choices made, $\mu$ will then be represented by an $e_1$-$e_2$-bimodule $m$. Let $f_i,g_i$ be the $1$-morphisms in $\Kar(\cC)$ presenting $Y_i$ as the condensate of $(X_i,e_i,\dots)$. The underlying $1$-morphism of $m$ is $g_1\mu f_2$. Furthermore, we have the following adjoint pairs in $\Kar(\cC)$: 
  $$ f_1 \dashv g_1, \qquad g_2 \dashv f_2, \qquad m \dashv m^R.$$
Including the morphism $m^R$ into diagram (\ref{eqn.bimod}), we find in $\Kar(\cC)$:
\begin{equation*}%\label{eqn.bimod}
\begin{tikzpicture}[baseline=(midpoint)]
  \path (0,0) node (X1) {$X_1$} +(0,-.35) coordinate (X1e) +(0,-2) node (Y1) {$Y_1$} +(0,-1) coordinate (midpoint);
  \draw[->] (X1) .. controls +(-.75,-.75) and +(-.75,.75) ..  node[auto,swap] {$\scriptstyle  f_1$} (Y1);
  \draw[->] (Y1) .. controls +(.75,.75) and +(.75,-.75) .. node[auto] {$\scriptstyle g_1$} (X1);
%  \draw[doublecondensation] (X1e) -- (Y1);
  \path (3,0) node (X2) {$X_2$} +(0,-.35) coordinate (X2e) +(0,-2) node (Y2) {$Y_2$};
  \draw[->] (X2) .. controls +(-.75,-.75) and +(-.75,.75) ..  node[auto] {$\scriptstyle f_2$} (Y2);
  \draw[->] (Y2) .. controls +(.75,.75) and +(.75,-.75) .. node[auto,swap] {$\scriptstyle g_2 $} (X2);
%  \draw[doublecondensation] (X2e) -- (Y2);
  \draw[->] (Y2) -- node[auto] {$\scriptstyle \mu$} (Y1);
  \path 
    (X2) +(-.5,-.25) coordinate (X2m)
    (Y2) +(-.5,.25) coordinate (Y2m)
    (X1) +(.5,-.25) coordinate (X1m)
    (Y1) +(.5,.25) coordinate (Y1m)
    ;
  \draw[->] (X2m) .. controls +(-.5,-.5) and +(-.5,.5) .. (Y2m) -- node[auto,swap] {$\scriptstyle m$} (Y1m) .. controls +(.5,.5) and +(.5,-.5) .. (X1m);
  \draw[->] (X1) -- node[auto] {$\scriptstyle m^R$} (X2);
  \path (1.5,-1) node {$\Leftarrow$} (-.125,-1) node {$\Downarrow$} (3.125,-1) node {$\Uparrow$};
\end{tikzpicture}
\end{equation*}
The $\Rightarrow$s stand in for both the units and counits of the various adjunctions.
  
  By construction, $f_1 m g_2 = f_1 g_1 \mu f_1 g_2$ carries two commuting condensation monads, whose product is a condensation monad on $f_1 m g_2$ with condensate $\mu$. By taking their mates, these determine commuting condensation monads supported by $(f_1 m g_2)^R = g_2^R m^R f_1^R = f_2 m^R g_1$. The condensate of their product is the right adjoint $\mu^R$ which we set out to construct.
  
  Finally, we need to show that if $k$-morphisms in $\cC$ admit adjoints for $k\geq 2$, then so do $k$-morphisms in $\Kar(\cC)$. A $k$-morphism in $\Kar(\cC)$ for $k\geq 2$ is a condensation bimodule morphism, which is a type of $(k-1)$-morphism in the category of functors from the walking condensation bimodule to $\cC$, which is a condensation bimodule in the category of $(k-1)$-morphisms in $\cC$. Thus the claim follows from what we have already proven because $k$-rigidity of $\cC$ implies $1$-rigidity of the category of
$(k-1)$-morphisms in $\cC$. (Compare \cite[Section~7]{JFS}, where dualizability and adjunctibility questions in categories of morphisms are studied in detail. Alternately, note that the right adjoint to any type of bimodule morphism is, by standard general nonsense, automatically a ``lax bimodule morphism'' in the sense that compatibility between the (co)actions and the morphism holds only up to not-necessarily-invertible higher morphisms; but in the presence of sufficient rigidity, lax bimodule morphisms are automatically strong \cite[Lemma 2.10]{DSPS2}.)
\end{proof}

By appealing to the Cobordism Hypothesis \cite{BaeDol95,Lur09}, we find:

\begin{corollary}\label{cor.tqft}
  Each object of $\Sigma^d \Vect_\bC$ determines a $(d+1)$-dimensional framed extended TQFT. Each $k$-morphism determines a codimension-$k$ defect (aka domain wall; see \cite[Example 4.3.23]{Lur09}) between framed extended TQFTs. The same statement holds with $\Vect_\bC$ replaced by $\SVect_\bC$, $\cat{Rep}(G)$, etc. \qed
\end{corollary}

 Together with Theorem~\ref{thm.Hamiltonian}, we find a very tight relationship between TQFTs and gapped condensed matter systems.

\subsection{Absolute colimits}\label{subsec.abcolim}

The theory of (co)limits in $n$-categories is subtle \cite{MR4519634}, and much work remains to be done. Our understanding is that Markus Zetto is currently in the process of developing a rigorous theory of absolute higher (co)limits; Zetto plans to rigorously prove the results sketched below, whereas we will take some assumptions about higher colimits for granted.

Very roughly, to say that some object $Y$ in an $n$-category $\cC$ is a \define{colimit} is to give a natural-in-$Z$ identification between morphisms $Y \to Z$ and some list of data of the form:
\begin{itemize}
  \item $1$-morphisms $x_i : X_i \to Z$, for some list of objects $X_i$;
  \item $2$-morphisms $\xi_j$ whose source and target $1$-morphisms are built by composing the $x_i$s with $1$-morphisms between the $X_i$s;
  \item $3$-morphisms whose source and target $2$-morphisms are built by composing the $\xi_j$s  with $2$-morphisms between the $1$-morphisms between the $X_i$s;
  \item and so on until $n$-morphisms;
  \item satisfying equations between those $n$-morphisms.
\end{itemize}
A \define{limit} is the same, except that one identifies morphisms $Z \to Y$ in terms of morphisms $Z \to X_i$ plus higher data. We will write $\cX = (X_i,\dots)$ for these data, and $Y = \colim \cX$ for its colimit. Given such data $\cX$, such a $Y$ is unique (up to a contractible space of choices) if it exists.

Suppose that $\cX = (X_i,\dots)$ is the data describing a colimit in some $n$-category $\cC$, and let $Y \in \cC$ be the corresponding colimit object. Suppose that $F : \cC \to \cD$ is a functor of $n$-categories. Then there is natural data $F(\cX)$ describing a colimit in $\cD$. The colimit of $F(\cX)$ in $\cD$, if it exists, may not be $F(Y)$. Rather, the universal description of maps $\colim F(\cX) \to Z$ implies that there is a canonical map $\colim F(\cX) \to F(\colim \cX)$ in $\cD$, which may not be an isomorphism. 
(If we had used limits instead, then there would be a canonical map $F(\lim \cX) \to \lim F(\cX)$.)
If it is an isomorphism, then $F$ is said to \define{preserve} the colimit described by $\cX$.
 The data $\cX$ describes an \define{absolute colimit} if the colimit described by $\cX$ is preserved by all functors.

\begin{example}\label{eg.colimitscondensation}
  Our main example is the condensate of a condensation monad. Indeed, suppose that $(X,e,\dots)$ is a condensation monad with condensate $Y$. The proof of Theorem~\ref{thm.uniqueness} identifies morphisms $Y \to Z$ and morphisms $Z \to Y$ with (right or left) condensation $e$-modules, and the data of condensation $e$-module is exactly of the form described above. So $Y$ is both a colimit and a limit.
  Furthermore, since
   functors take condensations to condensations, ``taking the condensate'' is absolute, both as a colimit and as a limit.
\end{example}

The $1$-categorical theory of colimits is well-developed and well-known.
For $2$-categories, the theory of colimits is less well-known, but it is well-developed (see~\cite{nLab-2limit} and references therein; note that though that the subtleties raised in \cite{MR4519634}  appear already in dimension $2$). Our imagined $n$-categorical colimits include all variants of $2$-categorical colimits, including what are usually called ``lax'' and ``oplax'' colimits --- the different variants restrict in different ways the allowed sources and targets of the $2$-morphisms $\xi_j$.

When $n>2$, the theory of colimits in $n$-categories is not, as of the time of this writing, particularly well-developed. 
The best way to do so will be to recognize that a weak $n$-category is the same as an $(\infty,1)$-category enriched in the $(\infty,1)$-category of weak $(n-1)$-categories \cite{MR2883823,MR3345192}. In the world of strict $1$-categories there is a well-developed notion of enriched, also called weighted, colimits \cite{MR2177301}. There will eventually be an analogous theory of enriched colimits in the $(\infty,1)$-world, and $n$-categorical colimits will be a main example. 

The $1$-categorical versions of the following facts provide the foundations of the theory of $1$-categorical colimits by recasting colimits in terms of Yoneda theory. Our philosophy will be to assume that they hold in the eventual theory of $n$-categorical colimits, building on the higher Yoneda theory developed by \cite{MR4080581}. In the meantime, we will take fact (\ref{fact.tiny}) as our definition of \define{Cauchy completion}.

\begin{enumerate}
  \item \label{fact.psh} Let $\cC$ be an $n$-category. A \define{presheaf} on $\cC$ is a functor of $n$-categories from $\cC$ to the $n$-category $\Cat_{n-1}$ of $(n-1)$-categories. Presheaves form an $n$-category called $\Psh(\cC)$. There is a fully faithful inclusion $\yo : \cC \to \Psh(\cC)$, called the \define{Yoneda embedding}, sending the object $Z \in \cC$ to the presheaf $\yo(Z) : Y \mapsto \hom(Y,Z)$. (The letter $\yo$ is the first letter of ``Yoneda'' when written in hiragana; its use for the Yoneda embedding was first suggested in \cite{JFS}.)
  For any presheaf $T$, there is a natural equivalence $\hom(\yo(X),T) = T(X)$.
  
  The Yoneda embedding presents $\Psh(\cC)$ as the universal closure of $\cC$ under colimits. Specifically, $\Psh(\cC)$ contains all  colimits, and if $\cD$ is any category containing all colimits, then functors $\cC \to \cD$ can be identified with colimit-preserving functors $\Psh(\cC) \to \cD$, where the identification takes the colimit-preserving functor $F : \Psh(\cC) \to \cD$ to the functor $F \circ \yo : \cC \to \cD$.
  
  \item \label{fact.tiny} A presheaf $S \in \Psh(\cC)$ is called \define{tiny} if the functor $\hom(S,-) : \Psh(\cC) \to \Cat_{n-1}$ preserves all colimits. (The name is due to \cite{MR3361309}, based on various ``small object'' notions in the category-theory literature.) The image of $\yo$ consists entirely of tiny objects, because if $(T_i,\dots)$ describes a colimit in $\Psh(\cC)$, then one can show directly that the presheaf $C \mapsto \colim (T_i(C),\dots)$ is its colimit.
  
  Write $\Cau(\cC)$ for the full sub-$n$-category of $\Psh(\cC)$ on the tiny objects. Then $\Cau(\cC)$ is the universal closure of $\cC$ under absolute colimits in the sense that if $\cD$ is any $n$-category containing all absolute colimits, then functors $\cC \to \cD$ and functors $\Cau(\cC)\to\cD$ are naturally identified. The $n$-category $\Cau(\cC)$ is called the \define{Cauchy completion} of $\cC$. (The name is due to \cite{MR0352214} and is based on an analogy between enriched categories and metric spaces.)
%  The Cauchy completion $\Cau(\cC)$ is also the universal closure of $\cC$ under absolute limits.
\end{enumerate}

Even without a complete description of $n$-categorical colimits, we claim:

\begin{theorem}\label{thm.absolute}
  Let $\cC$ be an $n$-category all of whose hom-$(n-1)$-categories have all condensates. Then $\Kar(\cC)$ and $\Cau(\cC)$ agree.
\end{theorem}

\begin{proof}
  As mentioned already in Example~\ref{eg.colimitscondensation}, taking the condensate of a condensation monad is an example of an absolute colimit. Abstract nonsense then provides a fully faithful inclusion $\Kar(\cC) \mono \Cau(\cC)$. For definiteness, we will describe this inclusion in detail.
  Consider $T \in \Psh(\Kar(\cC))$ a presheaf on $\Kar(\cC)$. By construction, each object $Y \in \Kar(\cC)$ is the condensate of a condensation monad $(X,e,\dots) \in \cC$. Apply $T$ to this condensation $X \condense Y$; the result is a condensation $TX \condense TY$ in $\Cat_{n-1}$. By Theorem~\ref{thm.uniqueness}, this condensation, and in particular $TY$, is determined by the condensation monad $(TX,Te,\dots)$. Thus restriction along $\cC \mono \Kar(\cC)$ gives a map $\Psh(\Kar(\cC)) \to \Psh(\cC)$, which is in fact an equivalence:  $\Cat_{n-1}$ contains all condensates (Example~\ref{eg.ncat-condcomplete}), and so for any $T \in \Psh(\cC)$ extends uniquely to a presheaf on $\Kar(\cC)$; the extension sends 
  $Y = (X,e,\dots) \in \Kar(\cC)$ to the condensate of $(TX, Te, \dots)$.
  But the Yoneda embedding $\yo : \Kar(\cC) \mono \Psh(\Kar(\cC)) = \Psh(\cC)$ takes values within the subcategory  of tiny objects.
  
   To establish the converse inclusion, we must show that for each tiny presheaf $S \in \Cau(\cC)$, there is a condensation $\yo(C) \condense S$ for some $C \in \cC$.
   
   Every presheaf can be written as a colimit of (the images under $\yo$ of) objects in $\cC$ --- this is part of the universality described in (\ref{fact.psh}) above.  The universal way to do this is as follows. The list of objects $X_i$ contains many copies of $\yo(C)$ for each $C\in\cC$, one copy for each element  $f \in \hom(\yo(C),S) = S(C)$; the map $x_i : X_i \to S$ is this $f : \yo(C) \to S$. For the 1-morphisms, one runs through all $1$-morphisms in $\cC$ and also all $1$-morphisms in the various $S(C)$s. Etcetera. Let us simply write $\cX = (X_i,\dots)$ for this list of data, so that $S = \colim \cX$.
   
   Now we use that $S$ is tiny. Then in particular $\hom(S,-)$ preserves the colimit of shape $\cX$:
   $$ \hom(S,S) = \hom(S, \colim \cX) = \colim \hom(S, \cX).$$
   (As with elsewhere in this paper, the symbol ``$=$'' does not mean equality of $(n-1)$-categories, but rather the presence of a canonical equivalence.)
   
   What is $\colim \hom(S, \cX)$; in particular, what are its objects? It is a colimit of $(n-1)$-categories. So let us consider some generic colimit of $(n-1)$-categories, described by some list of data $\cY = (Y_i, \dots)$. Then each $Y_i$ on the list is itself an $(n-1)$-category with a map to $\colim \cY$, and so the objects of the $Y_i$s provide objects of the colimit. These are not necessarily all of the objects of $\colim \cY$, as illustrated by Example~\ref{eg.colimitscondensation}: the objects of the condensate are the condensation modules, and not all modules are free.
   Rather, by unpacking the defining property of the colimit, one finds that every object in $\colim \cY$ is an absolute (!)\ colimit of objects in the image of the $Y_i$s. (Not all absolute such colimits need to exist in $\colim \cY$.)
   These are absolute colimits inside an $(n-1)$-category, and so by induction are simply condensations:
    for each each object in $\colim \cY$, there is an object in some $Y_i$ condensing to it.
      
   In particular, the identity map $\id_S \in \hom(S,S)$ must arise as the condensation of some object in some $\hom(S,X_i)$. Unpacked, this means that there is some pair $X_i = (C,f)$ with $C\in \cC$ and $f : \yo(C) \to S$, and some element $g \in \hom(S,\yo(C))$, such that the element $fg \in \hom(S,S)$ condenses to $\id_S$.
   Thus $\yo(C) \condense S$.
\end{proof}

Recall that, if $\cC$ is a monoidal $n$-category, then a \define{$\cC$-module $n$-category} is a module object for $\cC$ understood as an algebra object in the symmetric monoidal $(n+1)$-category $\Cat_{n}$. Equivalently, a $\cC$-module $n$-category is a functor of $(n+1)$-categories $\rB \cC \to \Cat_n$. If $\cC$ is symmetric monoidal, then so is $\cat{Mod}(\cC)$. (The tensor product is defined by a certain colimit in the $(\infty,1)$-category underlying $\Cat_n$ that we will not recall, and relies on the fact that, as an $(\infty,1)$-category, $\Cat_n$ contains all colimits and is Cartesian closed.)

\begin{corollary} \label{cor.dualizable}
  Let $\cC$ be a symmetric monoidal $n$-category with all condensates.  Let $\cat{Mod}(\cC)$ denote the symmetric monoidal $(n+1)$-category of all $\cC$-module $n$-categories.  Its full subcategory on the $1$-dualizable objects is equivalent to $\Sigma \cC$.
\end{corollary}

\begin{proof}
  By definition, $\cat{Mod}(\cC) = \Psh(\rB\cC)$.  By Theorem~\ref{thm.absolute}, $\Sigma \cC = \Kar(\rB\cC)$ is the full sub-$(n+1)$-category of $\cat{Mod}(\cC)$ on the tiny objects.
  In any category of modules, an object is tiny if and only if it is $1$-dualizable. In one direction, if  $M \in \cat{Mod}(\cC)$  has a dual module $M^*$, then $\hom_\cC(M,-)$ is equivalent to $M^*\otimes_\cC (-)$ and so preserves colimits. In the other direction, if $M$ is tiny, then $M^* = \hom_\cC(M,\cC)$ is its dual object: the functors $\hom_\cC(M,-)$ and $M^* \otimes_\cC (-)$ take the same value on $\cC \in \cat{Mod}(\cC)$ and both preserve colimits, and so agree.
\end{proof}

Combined with Theorem~\ref{thm.fullrigidity}, we learn the following remarkable fact, generalizing the characterization of $1$-dualizable linear categories from \cite[Section 2]{MR1670122}:

\begin{corollary} \label{cor.fullydualizable}
  Suppose that $\cC$ is a fully-rigid symmetric monoidal $n$-category with all condensates. Then a $\cC$-module $n$-category is $1$-dualizable if and only if it is fully-dualizable. \qed
\end{corollary}

\subsection{Direct sums and additive $n$-categories}

Suppose that $\cC$ is a higher category equipped with some notion of ``addition''  of $1$-morphism. A \define{direct sum} of objects $X,Y \in \cC$ is an object $X \oplus Y$ equipped with maps
$$ f_X : X \oplus Y \to X, \quad f_Y : X \oplus Y \to Y, \quad g_X : X \to X \oplus Y, \quad g_Y : Y \to X \oplus Y, $$
such that
$$ g_X f_X + g_Y f_Y \cong \id_{X \oplus Y}, \quad f_A g_B \cong \begin{cases} \id_A, & A = B, \\ 0, & A \neq B. \end{cases} $$
Such $X \oplus Y$ is then both the product and coproduct of $X$ with $Y$. The category $\cC$ \define{has direct sums} if it has a \define{zero object} --- an object $0 \in \cC$ such that $\id_0 = 0$ --- and if every pair of objects $X,Y$ admits a direct sum $X \oplus Y$. %(The direct sum and the zero object are each unique up to a contractible space of choices.)

\begin{remark}
  Our discussion is intentionally slightly sloppy about the difference, vital in higher category theory, between data and property. Specifically, should the choice of isomorphism be part of the data of the direct sum? The answer is that if all isomorphisms $g_X f_X + g_Y f_Y \cong \id_{X \oplus Y}$, $f_X g_X \cong \id_X$, etc., are part of the data of $X \oplus Y$, then one should also include higher coherence data relating these choices of isomorphisms. A more compact option is to judiciously choose some of the isomorphisms as data, and declare the existence of the others as property.
\end{remark}

Direct sum is itself a form of ``addition.'' As such, the most natural way for an $n$-category, with $n\geq 2$, to have an ``addition'' on $1$-morphisms is if it has an addition on $2$-morphisms and the addition on $1$-morphisms is declared to be the direct sum of $1$-morphisms. One is led naturally to the following notion. An \define{$\cat{Ab}$-enriched $n$-category} is an $n$-category such that the sets of $n$-morphisms (recall from the start of Section~\ref{sec.Kar} that we work with ``weak $n$-'' rather than ``$(\infty,n)$-'' categories, so that $n$-morphisms do form sets) are given the structure of abelian groups, and such that all compositions are multilinear on $n$-morphisms. The abelian group structure on $n$-morphisms determines a notion of direct sum on $(n-1)$-morphisms, which determines a notion of direct sum on $(n-2)$-morphisms, all the way to objects. An \define{additive $n$-category} is an $\cat{Ab}$-enriched $n$-category which has direct sums at all levels.

\begin{lemma}\label{lemma.sigmadirectsum}
Let $\cC$ be a monoidal additive $(n-1)$-category with condensates and direct sums.  Then $\Sigma\cC = \Kar(\rB \cC)$, which is by construction an $\cat{Ab}$-enriched $n$-category with condensates, also has direct sums.
\end{lemma}
\begin{proof}
Suppose objects $A,B \in \Sigma\cC$ are represented by condensation algebras, also denoted $A$ and $B$, in $\cC$. Then the direct sum $A \oplus B$ in $\cC$ carries a distinguished condensation algebra structure, and it represents the direct sum of objects in $\Sigma \cC$.
\end{proof}

Lemma~\ref{lemma.sigmadirectsum} relies on the fact that $\rB\cC$ has only one object. If it were replaced by an $n$-category with multiple objects, then condensation monads living at different objects would not be direct-summable in the Karoubi completion. Instead, their direct sum must be included by hand. In general,
suppose $\cC$ is an additive $n$-category whose hom $(n-1)$-categories have direct sums. The \define{matrix category} of $\cC$, denoted $\Mat(\cC)$, is the $n$-category whose objects are vectors (i.e.\ finite sequences) of objects in $\cC$, and whose morphisms are matrices of morphisms in $\cC$. The matrix category is precisely the closure of $\cC$ under direct sums. It is not hard to show furthermore that $\Kar(\Mat(\cC))$ has direct sums, and so the combination $\Kar\circ\Mat$ is the closure under both direct sums and condensates.

Indeed, the results of \S\ref{subsec.abcolim} apply essentially verbatim in the additive world, provided $\Kar$ is replaced by $\Kar\circ \Mat$. Indeed, an $\cat{Ab}$-enriched $n$-category is nothing but an $(\infty,1)$-category enriched, in the $(\infty,1)$-sense, in $\cat{Ab}$-enriched $(n-1)$-categories, and so the eventual development of a theory of enriched colimits in the $(\infty,1)$-world will include the case of colimits in $\cat{Ab}$-enriched $n$-categories.
$\cat{Ab}$-enriched colimits include colimits, but one may also add $k$-morphisms together when giving the data $\cX$; an $\cat{Ab}$-enriched colimit is absolute if it is preserved by all $\cat{Ab}$-enriched functors.
\begin{example}
  Finite direct sums are $\cat{Ab}$-enriched-absolute.
\end{example}
$\cat{Ab}$-enriched $(n-1)$-categories naturally form an $\cat{Ab}$-enriched $n$-category $\cat{AbCat}_{n-1}$, and an \define{$\cat{Ab}$-enriched presheaf} on an  $\cat{Ab}$-enriched $n$-category $\cC$ is a functor of $\cat{Ab}$-enriched $n$-categories $\cC \to \cat{AbCat}_{n-1}$. The $\cat{Ab}$-enriched $n$-category of $\cat{Ab}$-enriched presheaves will be the universal closure under $\cat{Ab}$-enriched colimits. It still makes sense to talk about tiny objects and the Cauchy completion: in the $\cat{Ab}$-enriched world, an object $S$ is \define{tiny} if the $\cat{AbCat}_{n-1}$-valued functor $\hom(S,-)$ preserves colimits.

To modify Theorem~\ref{thm.absolute} to the $\cat{Ab}$-enriched setting, one must make only one change: when studying a colimit $\colim \cY$ of $(n-1)$-categories, in the $\cat{Ab}$-enriched case the objects of the colimit are absolute-in-the-$\cat{Ab}$-enriched-sense objects in the image of the $Y_i$s, and so by induction are condensations of direct sums, rather than merely condensations. From this, one finds $\id_S$ as a condensation of a direct sum of objects in $\cC$. Thus:
\begin{theorem}
  The Cauchy completion among $\cat{Ab}$-enriched $n$-categories is $\Kar(\Mat(-))$. \qed
\end{theorem}

Finally, in light of Lemma~\ref{lemma.sigmadirectsum}, we recommend the following notation. Suppose $\cC$ is monoidal with a notion of addition which may not be the direct sum. Then $\Sigma\cC$ will mean $\Kar(\Mat(\rB\cC))$ rather than merely $\Kar(\rB\cC)$.
For additive $n$-categories, this meaning of $\Sigma$ agrees with our earlier meaning except when $n=0$. An additive $0$-category is an abelian group, and a monoidal additive $0$-category is a ring $R$, in which case $\Sigma R = \Kar(\Mat(\rB R))$ is precisely the category of finitely generated projective $R$-modules. With this notation, Corollaries~\ref{cor.dualizable} and~\ref{cor.fullydualizable}, and their proofs, apply verbatim with the words ``additive'' included, even when $n=0$.

\begin{example}
  If $\bK$ is a field, then $\Vect_\bK \simeq \Sigma \bK$, where as always in this paper ``$\Vect_\bK$'' means the category of finite-dimensional and $\Sigma \bK$ means the additive version. Then Corollary~\ref{cor.tqft} could be phrased as ``Each object of $\Sigma^n \bC$ determines an $n$-dimensional extended TQFT,'' and Theorem~\ref{thm.Hamiltonian} could be phrased as ``Each object of $\Sigma^n\bC$ determines an $n$-dimensional commuting projector Hamiltonian system,'' where in both cases $n = (n-1)+1$ is the spacetime dimension.  Corollaries~\ref{cor.dualizable} and~\ref{cor.fullydualizable}, when applied to $\cC=\Sigma^n \bK$, say that a $\bK$-linear $n$-category is $1$-dualizable, in the $(n+1)$-category of all Cauchy-complete $\bK$-linear $n$-categories, if and only if it is fully dualizale, and that the $(n+1)$-category of dualizable $\bK$-linear $n$-categories is precisely to $\Sigma^{n+1}\bK$.
\end{example}

\bibliography{lattice}{}
\bibliographystyle{alpha}

\end{document}